\documentclass[twoside,11pt, reqno]{amsart}

\usepackage[top=1.1in,bottom=1.1in,left=1.2in,right=1.2in]{geometry}
\usepackage{amsmath,amssymb,amsthm,amsfonts,enumerate, mathrsfs, mathtools, appendix}
\usepackage{stmaryrd}
\usepackage{esint}
\usepackage{xcolor}
\usepackage{graphicx}
\usepackage{tikz} 
\usetikzlibrary{patterns,positioning,arrows,shapes,trees}

\usepackage[
colorlinks=true,
linkcolor=blue,
filecolor=blue,
urlcolor=magenta,
citecolor=magenta]{hyperref}

\usepackage{fancyhdr}
{} 

\makeatletter
\newcommand{\mylabel}[2]{#2\def\@currentlabel{#2}\label{#1}}
\makeatother

\newcounter{counterConstant}

\numberwithin{equation}{section}
\def\theequation{\arabic{section}.\arabic{equation}}

\newtheorem{theorem}{Theorem}[section]
\newtheorem{lemma}[theorem]{Lemma}
\newtheorem{remark}[theorem]{Remark}
\newtheorem{definition}[theorem]{Definition}
\newtheorem{proposition}[theorem]{Proposition}
\newtheorem{corollary}[theorem]{Corollary}
\newtheorem{assumption}{Assumption}
\newtheorem{example}{Example}

\def\cA{{\mathcal A}}

\def\cC{{\mathcal C}}
\def\cD{{\mathcal D}}

\def\cF{{\mathcal F}}

\def\cL{{\mathcal L}}

\def\cP{{\mathcal P}}

\def\cR{{\mathcal R}}

\def\mE{{\mathbb E}}

\def\mN{{\mathbb N}}

\def\mP{{\mathbb P}}
\def\mQ{{\mathbb Q}}
\def\mR{{\mathbb R}}
\def\mS{{\mathbb S}}

\def\bE{{\mathbf E}}

\def\bP{{\mathbf P}}

\def\sC{{\mathscr C}}

\def\sF{{\mathscr F}}

\def\sL{{\mathscr L}}

\def\sP{{\mathscr P}}

\def\sS{{\mathscr S}}

\def\sX{{\mathscr X}}

\def\l{\left}
\def\r{\right}
\def\<{\langle}
\def\>{\rangle}

\def\geq{\geqslant}
\def\leq{\leqslant}

\def\1{{\mathbf{1}}}
\def\p{\partial}
\def\d{\text{\rm{d}}}
\def\e{\mathrm{e}}
\def\eps{\varepsilon}
\def\Re{{\mathrm{Re}}}

\allowdisplaybreaks

\begin{document}

\title{Non-local operators with low singularity kernels: regularity estimates and martingale problem}

\author{Eryan Hu and Guohuan Zhao}

\address{Center for Applied Mathematics and KL-AAGDM, Tianjin University, Tianjin, 300072, China}
\email{eryan.hu@tju.edu.cn}

\address{Institute of Applied Mathematics, Academy of Mathematics and Systems Science, CAS, Beijing, 100190, China}
\email{gzhao@amss.ac.cn}

\begin{abstract} 
We consider the linear non-local operator $\mathcal{L}$ denoted by
\[
\mathcal{L} u (x) = \int_{\mathbb{R}^d} \left(u(x+z)-u(x)\right) a(x,z)J(z)\,\mathrm{d} z. 
\]
Here $a(x,z)$ is bounded and $J(z)$ is the jump kernel of a L\'evy process, which only has a low-order singularity near the origin and does not allow for standard scaling. The aim of this work is twofold. Firstly, we introduce generalized Orlicz-Besov spaces tailored to accommodate the analysis of elliptic equations associated with $\mathcal{L}$, and establish regularity results for the solutions of such equations in these spaces. Secondly, we investigate the martingale problem associated with $\mathcal{L}$. By utilizing analytic results, we prove the well-posedness of the martingale problem under mild conditions. Finally, we obtain a new Krylov-type estimate for the martingale solution through the use of a Morrey-type inequality for generalized Orlicz-Besov spaces.
\end{abstract}

\maketitle

\tableofcontents


\section{Introduction} 

Over the past two decades, extensive research efforts have been dedicated to investigating a class of non-local operators related to jump Markov processes, called $\alpha$ stable-like operators, which can be formulated as: 
\begin{equation}\label{eq:stable}
	\sL u (x)= \int_{\mR^d} \l(u(x+z)-u(x)- \nabla u(x)\cdot z \1_{B_1}(z)\r)  \frac{a(x,z)}{|z|^{d+\alpha}}\,\d z, \quad \alpha\in (0,2)
\end{equation}  
(cf. \cite{bass1988uniqueness}). Here $a: \mR^d\times\mR^d\to \mR$ is a bounded measurable function, and $u\in C_b^\infty$.  When $a$ is smooth, it is known that there exists a pure jump Markov process $X$ generated by $\sL$, and $a(x,z)|z|^{-d-\alpha}$ is the intensity of jumps from $x$ to $x+z$ (cf. \cite{jacob2005pseudo3}).

The exploration of regularity properties for linear and non-linear  equations corresponding to the above  non-local operators has been the subject of extensive investigation. In particular, Schauder estimates for linear equations were investigated in \cite{bae2015schauder, bass2009regularity, dong2013schauder,   ling2022nonlocal, mikulevicius2014cauchy}, while the non-linear cases were studied in \cite{ dong2018dini, dong2017dini}; $L^p$-estimates for non-local equations are available in \cite{dong2023boundedness, dong2012lpestimates, mikulevivcius1992cauchy, zhang2013maximal};  Harnack inequalities for positive $\sL$-harmonic functions and H\"older continuity of $\sL$-harmonic functions were first established in \cite{bass2002harnack} under the assumption that the function $a$ satisfies certain boundedness conditions. For more related results, readers can refer to \cite{caffarelli2011regularity,caffarelli2009regularity, chen2003heat, dyda2020regularity, kassmann2009priori} and the references therein.  

In the context of jump Markov processes, several studies have focused on the martingale problem for $\sL$ given by \eqref{eq:stable}, including \cite{abels2009cauchy, bass2009martingale, chen2016uniqueness, mikulevicius2014cauchy}, among others. Related problems involving stochastic differential equations driven by rectilinear $\alpha$-stable processes were also explored in \cite{bass2006systems} and \cite{chen2021supercritical}.

\medskip

In this paper, we consider a class of non-local operators including the following one: 
\begin{equation}\label{Eq:Lu}
  L u (x) = \int_{\mR^d} \l(u(x+z)-u(x)\r) a(x,z) \frac{\1_{B_1}(z)}{|z|^d}\,\d z. 
\end{equation}
Here $a:\mR^d\times \mR^d\to \mR$ is a non-negative bounded measurable function and $B_1$ is the unit ball in $\mR^d$. 

Comparing  \eqref{eq:stable} and \eqref{Eq:Lu}, one can roughly think of \eqref{Eq:Lu} as being about the ``$\alpha=0$" case, which turns out to be a more challenging situation. Although there is limited research on this type of operators compared to the stable case, we have noted that several significant papers have examined issues related to operators like \eqref{Eq:Lu} from various perspectives, highlighting their important applications in different fields. Notably,  {\v{S}}iki{\'c}-Song-Vondra{\v{c}}ek \cite{sikic2006potential} first studied the potential theory of  subordinate Brownian motions with geometric stable subordinators, whose generator is given by $\log(I+(-\Delta)^{\alpha/2})$ (below we will see that $\log(I+(-\Delta)^{\alpha/2})$  is a variant of $L$).  They established the asymptotic behaviors of the Green function and the L\'evy density of these processes.  Among the PDE articles, Chen-Weth \cite{chen2019dirichlet} demonstrated that 
\begin{equation*}
	\begin{aligned}
		&\frac{\d }{\d s}(-\Delta)^s u(x) \Big|_{s=0} = \log (-\Delta) u(x)\\
		=& -c_d\int_{\mR^d} (u(x+z)-u(x)) \frac{\1_{B_1}(z)}{|z|^d}~\d z-c_d\int_{|z|\geq 1} \frac{u(x+z)}{|z|^d}~\d z+ \rho_d u(x), 
	\end{aligned}
\end{equation*}
where $c_d$ and $\rho_d$ are constants only depending on $d$. So the dominating term of $\frac{\d }{\d s}(-\Delta)^s\big|_{s=0}$ is $-c_d L$, as defined in \eqref{Eq:Lu} with $a=1$. In the same work, the authors studied the eigenvalue of Dirichlet problem for the logarithmic Laplacian. Later, Jarohs-Salda\~na-Weth \cite{jarohs2020new} offered a novel look at the fractional Poisson problem via the logarithmic Laplacian given above.  We also mention that there is some literature revealing the  importance of logarithmic regularization operator $T:= \log^{-1}(eI-\Delta)$ ($T^{-1}$ is also a variant of $L$). For instance, $T$ is widely used in the studies of the local (global) well-posedness of the regularized inviscid models, such as the regularized $2$D  Euler equation, the surface quasi-geostrophic equation, and the Boussinesq equations  (see \cite{chae2012logarithmically} and the references therein).

\smallskip

Our work is partially inspired by  Kassmann-Mimica\cite{kassmann2017intrinsic}, where Kassmann and Mimica  considered a much larger class of non-local operators including \eqref{eq:stable} and \eqref{Eq:Lu}.  To overcome the difficulties arising from the lack of the scaling invariant property of $L$, they introduced a new intrinsic scale, which can be used to establish a growth lemma for $L$ and a modification of hitting probability estimate for Markov process corresponding to $L$. As an application, $\log^{-\beta}(1/r)$-order (for some $0<\beta<1$) regularity estimate for solutions to non-local elliptic equation $L u=f$ was obtained under the assumption that the coefficient $a$ only satisfies some boundedness conditions (see also \cite{mimica2014harmonic} for the case $a=1$). In parallel, the study conducted by Chen-Zhang \cite{chen2014holder} delves into parabolic equations with drift terms. However, an unresolved question remains: can the continuity modulus of $u$ reach $\log^{-1}(1/r)$ when we impose continuity conditions on $a$, along with the condition that $Lu\in L^\infty$? 

On the other hand, while the well-posedness of the martingale problem for the $\alpha$ stable-like operator $\sL$ ($0<\alpha<2$) as defined in equation \eqref{eq:stable} has been extensively studied in various literature, to the best of our knowledge, there is a notable absence of research on the martingale problem for $L$. However, the existence and uniqueness of martingale solutions associated with $L$ appear as premise assumptions in some works, such as \cite{chen2014holder} and \cite{kassmann2017intrinsic}. So it is necessary to provide sufficient conditions for the establishment of these assumptions. 

Therefore, in this paper, we aim to address the following two questions: 
\begin{enumerate}[({\bf Q}$_1$)]
	\item Does the solution $u$ to the equation $L u = f$ exhibit additional smoothness under certain continuity conditions satisfied by $a$? 
	\item  Is it possible to identify a unique strong Markov process $X$, for which the infinitesimal generator matches $L$ when acting upon smooth functions, even when $a$ only meets some minimal conditions? 
\end{enumerate}

Recall that, for a second order, uniformly elliptic operator $A u=\sum_{i,j} a_{ij}\p_{ij}u$, the standard Schauder estimates and $L^p$-estimates are as follows: 
\begin{enumerate}[(a)]
	\item For any $\alpha\in (0,1)$, there is a constant $C$ such that for each $u\in C^{2,\alpha}$,
	\begin{equation*}
		\|u\|_{C^{2+\alpha}} \leq C \l(\|A u\|_{C^{\alpha}}+ \|u\|_\infty\r), 
	\end{equation*}
	provided that $a\in C^\alpha$; 
	\item For any $p\in (1,\infty)$, there is a constant $C$ such that for each $u\in W^{2}_{p}$, 
\end{enumerate}
\begin{equation*}
	\|u\|_{W^{2}_{p}} \leq C \l(\|A u\|_p+\|u\|_p\r), 
\end{equation*}
provided that $a$ is uniformly continuous (see \cite{han2011elliptic}). Both of these results are useful in the study of diffusion processes. In particular, (b) can be directly used to prove the uniqueness of the martingale problem corresponding to $A$ when $a$ is merely continuous (see Stroock and Varadhan \cite{stroock2007multidimensional}). 

\smallskip

Following the path of Stroock and Varadhan, in response to ({\bf Q}$_1$) and ({\bf Q}$_2$), in this work, we 
\begin{itemize}
	\item offer a series of analytical results for the elliptic equation 
	$$
	\lambda u-Lu=f  
	$$
	encompassing Schauder-type estimates (see Theorem \ref{Thm:Main1} (a) below). 
	\item construct the strong Markov process associated with $L$ (see Theorem \ref{Thm:Main2}). 
	\item provide an a priori estimate for the solutions to the corresponding  elliptic equation within a generalized Orlicz-Besov space (see Lemma \ref{Le-Main1}), facilitating the derivation of a Krylov-type estimate (or occupation time estimate) for the jump processes (see Theorem \ref{Thm:Main1} (b)). 
\end{itemize} 

\subsection{Main results and examples}
Although an important example of an operator we have in mind is \eqref{Eq:Lu}, we will consider a more general class of non-local operators that naturally arise from jump Markov processes. To state our main theorems, we must first introduce some fundamental concepts from probability theory. Specifically, let \( S = (S_t)_{t \geq 0} \) denote a driftless subordinator, which is a L\'evy process with non-decreasing paths and non-negative jumps. Its behavior is fully characterized by its L\'evy measure \( \Pi \), defined on \( (0, \infty) \), satisfying the condition \( \int_{0}^{\infty} (1\wedge t) \, \Pi(\d t) < \infty \). We assume that the Laplace exponent for \( S \) is given by 
\begin{equation}\label{eq:rept-phi}
	\phi(s)= -\log \bE \e^{-s S_1}=\int_0^\infty (1-\e^{-s t}) \Pi(\d t)= s \int_0^\infty \e^{-s t}  \Pi((t,\infty))\d t. 
\end{equation}
Let $B$ be a $d$-dimensional Brownian motion independent of $S$. Set 
\[
Z_t = \sqrt{2}B_{S_t}, 
\]
which is a pure jump L\'evy process known as the subordinate Brownian motion corresponding to $S$. Its jump kernel $J(z)$ is rotationally symmetric, meaning that there is a function $j:(0,\infty)\to (0,\infty)$ such that 
\begin{equation}\label{Eq:j}
    J(z)=j(|z|)= \int_{0}^{\infty} (4\pi t)^{-\frac{d}{2}}\e^{-\frac{|z|^2}{4t}}~ \Pi(\d t). 
\end{equation}
The L\'evy exponent of $Z$ is given by the formula
\begin{equation*}
	\psi(\xi):=-\log \bE \e^{i \xi\cdot Z_1}=\phi(|\xi|^2).
\end{equation*}
With a slight abuse of notation, the function $\psi$ can be viewed as a function from $\mathbb{R}_+$ to $\mathbb{R}_+$. Given $\psi$, we put 
\begin{equation}\label{Eq:rho}
	\rho(r) = \rho_\psi(r):= \frac{1}{\psi(r^{-1})}. 
\end{equation} 

We study the following spatial inhomogeneous non-local operator: 
\begin{equation}\label{Eq:Oprt}
	\cL u = \int_{\mR^d} \l(u(x+z)-u(x)\r)\,a(x,z)J(z)\,\d z.
\end{equation}

Our assumptions on $\phi$ (or $\Pi$) and $a$  are 
\begin{assumption}
	\begin{enumerate}
        \item[\mylabel{Aspt:J-1}{$\bf (\Phi)$}]: 
        $\phi$ is a slowly varying function at infinity (see Definition \ref{Def-SVF}) and $\lim_{s\to\infty}\phi(s)=\infty$; 
		\item[\mylabel{Aspt:a-1}{$\bf (A_{1})$}]:
		There are positive real numbers $\rho_0, c_0 \in (0,1)$ such that 
		\begin{equation}
			a(x,z)\geq c_0, \ \mbox{ for every }x\in \mR^d, z\in B_{\rho_0}=\{x\in \mR^d: |x|<\rho_0\} \tag{L$_{a}$} \label{aspt:a-lower}; 
		\end{equation}
		\item[\mylabel{Aspt:a-2}{$\bf (A_{2})$}]: There are real numbers $\alpha>0, c_0\in (0,1)$ such that  
		\begin{equation}
			\|a(\cdot,z)\|_{C^{\alpha}_{\rho}}\leq c_0^{-1}, \ \mbox{ for every } z\in \mR^d,  \tag{H$_{a}$} \label{Aspt:a-holder}
		\end{equation}
		where $\rho$ is given by \eqref{Eq:rho} and $C^\alpha_\rho$ is the generalized H\"older-type space defined in Section \ref{Sec-OB} (Definition \ref{Def-Holder}). 
	\end{enumerate}
\end{assumption}

We emphasize that many important examples satisfy the above conditions, such as the core examples in this paper: \(\log(I-\Delta)\) and \(L\) in \eqref{Eq:Lu}. 

We now can formulate our first main result: 
\begin{theorem}[Schauder-type estimates]\label{Thm:Main1}
    Suppose that $\phi$ satisfies \ref{Aspt:J-1}, and $a$ satisfies \ref{Aspt:a-1}-\ref{Aspt:a-2}. Then there exists a constant $\lambda_0>0$ such that for each $\lambda\geq 2\lambda_0$ and $f\in C_{\rho}^{\alpha}$, the following Poisson equation
	\begin{equation}\label{Eq:PE}
		\lambda u-\cL u =f \tag{PE}
	\end{equation}
	admits a unique solution in $C_{\rho}^{1+\alpha}$. Moreover, for any $\beta\in [0,\alpha]$,

	\begin{equation*}
		\lambda \|u\|_{C_{\rho}^{\beta}}+\|u\|_{C_{\rho}^{1+\beta}} \leq C \|f\|_{C_{\rho}^{\beta}}, 
	\end{equation*}
	where $C$ only depends on $d, \psi, \rho_0, c_0$, $\alpha$ and $\beta$. 
 \end{theorem}

 \begin{remark}
	\begin{enumerate}[(i)]
		\item 
        The method employed in this article is quite robust. With some modifications, most of the results can be readily extended to the case where $\phi$ is a general regularly varying function (see Definition \ref{Def-SVF}). However, for the sake of readability of the article, we only consider the case where $\phi$ is slowly varying at infinity. 
		\item 
		$\rho(|x-y|)$ can be regarded as the ``intrinsic" distance between $x$ and $y$, corresponding to the operator $\cL$. Consider the example where $\phi(s)=\sqrt{s}$, then  $\cL=-\sqrt{-\Delta}$ and $\rho(|x-y|)$ represents the usual Euclidean distance. In this context, the result analogous to Theorem \ref{Thm:Main1} is: for any $\alpha\in(0,1)$, 
		\[
		\|u\|_{C^{1+\alpha}} \leq C \l( \|\sqrt{-\Delta} u\|_{C^{\alpha}}+ \|u\|_{L^\infty} \r), 
		\]
		which was proved by Bass in \cite{bass2009regularity}. 
		\item 
		In this paper, we also establish some estimates for \eqref{Eq:PE} in Orlicz-Besov spaces, considering cases where $a$ is spatially homogeneous or satisfies a small oscillation condition. Refer to Theorem \ref{Thm:Main0} and Lemma \ref{Le-Main1} for details.
	\end{enumerate}
\end{remark}

Theorem \ref{Thm:Main1} implies that 
\begin{corollary}\label{Cor:Log}
	Let $c_0\in (0,1)$, $\alpha>0$ and $\rho(r)=1/\log(1+\frac{1}{r})$. Suppose that $c_0\leq a\leq c_0^{-1}$ and $|a(x,z)-a(x',z)|\leq c_0^{-1} \rho^\alpha(|x-x'|)$, then for each $\beta\in [0,\alpha]$, it holds that 
	\begin{equation*}
		\|u\|_{C_{\rho}^{1+\beta}} \leq C \l(\|L u\|_{C_{\rho}^{\beta}}+\|u\|_{\infty}\r), \quad \forall u\in C_{\rho}^{1+\beta}. 
	\end{equation*}
	Here $L$ is the operator defined in \eqref{Eq:Lu}, and $C$ only depends on $d, \rho_0, c_0$, $\alpha$ and $\beta$. 
\end{corollary}

In Section \ref{subsec:subordinator}, we will verify that the following key examples satisfy our assumptions. 
\begin{example}\label{ex:gamma}
    One notable example of \(S\) is the gamma subordinator, whose L\'evy measure, density and Laplace exponent are given by 
    \[
        \Pi(\d t) = \e^{-t}t^{-1} \d t, ~ \bP(S_t=u) = \frac{1}{\Gamma(t)}u^{t-1} \e^{-u}  ~\mbox{ and }~  \phi(s)=\log (1+s), 
    \]
    respectively. In this case, the process $Z=\sqrt{2}B_{S_t}$ is called the variance gamma process whose infinitesimal generator is 
    \[
	\cL u(x)= -\log(I-\Delta) u(x) =\int_{\mR^d} \l(u(x+z)-u(x)\r)\frac{\d z}{|z|^d\ell(|z|)}.
    \]
    Here $\ell:(0,\infty)\to (0,\infty)$ satisfies ${\ell(r)}\asymp 1~(r\to 0^+)$ (see more details after Proposition \ref{Prop:j-regular}). 
\end{example}

\begin{example}\label{ex:sde}
    Another example is the infinitesimal generator of the solution to stochastic differential equation (SDE) driven by {\em variance gamma process} $Z_t=\sqrt{2}B_{S_t}$:
	\begin{equation}\label{Eq:SDE}
		\d X_t =\sigma (X_{t-})\,\d Z_t, \quad X_0=x.
	\end{equation}
    Assume that $\sigma: \mR^d\to \mR^{d\times d}$ is invertible. Then the infinitesimal generator $\cL_\sigma$ of $X$ is formulated by 
    \begin{equation}\label{Eq:L-sde}
	\begin{aligned}
            \cL_\sigma u(x)=&\int_{\mR^d} \l(u(x+\sigma(x)z)-u(x)\r) \underbrace{\frac{1}{|z|^d\ell(|z|)}}_{J(z)}~\d z\\
            =& \int_{\mR^d} (u(x+z)-u(x)) \underbrace{ \frac{|z|^d\ell(|z|)}{|\det \sigma(x)|\cdot|\sigma^{-1}(x)z|^{d}  \ell(|\sigma^{-1}(x)z|)}}_{a(x,z)} J(z) \d z. 
	\end{aligned}
    \end{equation}
    When \(\sigma\) is \(\rho^\alpha\)-continuous and \(c_0 |\xi| \leq |\sigma \xi|\leq c_0^{-1}|\xi|\), for all \(\xi\in\mR^d\), then \(a(x,z)\) satisfies \ref{Aspt:a-1}-\ref{Aspt:a-2} (see the discussion after Proposition \ref{Prop:j-regular}). 
\end{example}

Although we consider general non-local operators with low singularity kernels in our study, it is worth mentioning that one of examples we have in mind is the operator $L$ defined in \eqref{Eq:Lu}. 
\begin{example}\label{ex:logLaplace}
	We can rewrite $L$ in \eqref{Eq:Lu} as 
	\[
	Lu(x)= \int_{\mR^d} (u(x+z)-u(x))\underbrace{a(x,z)\1_{B_1}(z)\ell(|z|)}_{\widetilde{a}(x,z)} \underbrace{\frac{1}{|z|^d\ell(|z|)}}_{J(z)}~\d z,  
	\]
	where $J$ is the jump kernel of the variance gamma process given in Example \ref{ex:gamma}. $\widetilde{a}$ satisfies \ref{Aspt:a-1}-\ref{Aspt:a-2} if $a$ satisfies the same conditions, due to the fact that $\ell(r)$ is bounded from below and above (see also the discussion after Proposition \ref{Prop:j-regular}). 
\end{example}

Our second main result concerns the weak well-posedness of equation \eqref{Eq:SDE} and relies on the martingale problem associated with non-local operators of the form \eqref{Eq:Oprt}, which includes $\cL_\sigma$.

\begin{theorem}\label{Thm:Main2}
    Suppose \ref{Aspt:J-1} and \ref{Aspt:a-1}-\ref{Aspt:a-2}  are satisfied. Then  
	\begin{enumerate}[(a)]
		\item for each $x\in \mR^d$, the martingale problem $(\cL, \delta_x)$ has a unique solution $\mP_x$, and the family $(\mP_x,X)_{x\in\mR^d}$ forms a strong Markov process on $\mR^d$; 
		\item  for any $N$-function $A$ (see Definition \ref{Def-Young}) satisfying $A(t)\gtrsim_\eps  [\psi^{-1}(t^{1+\eps})]^d$ $(\forall \eps>0)$, the following Krylov-type estimate is valid
		\begin{equation}\label{Eq:Krylov1}
			\begin{aligned}
				\mE_x \int_0^\infty \e^{-\lambda t} f(X_t)\,\d t\leq C \|f\|_A/\lambda, \quad \lambda>0 \mbox{ and } f\in L_A.  
			\end{aligned}  
		\end{equation}
		Here $\|f\|_A$ is the Luxemburg norm of $f$ with respect to $A$ (see Definition \ref{Def-Luxemburg}), and $C$ does not depend on $f$.  
	\end{enumerate}  
\end{theorem}

\begin{corollary}\label{Cor:SDE}
    Let \(Z\) be the variance gamma process. Let $c_0\in (0,1)$, $\alpha>0$ and $\rho(r)=1/\log(1+r^{-1})$. Suppose \(c_0 |\xi| \leq |\sigma \xi|\leq c_0^{-1}|\xi|\), for all \(\xi\in\mR^d\), and 
    \[
    |\sigma(x)-\sigma(x')|\leq c_0^{-1} \rho^\alpha(|x-x'|). 
    \]
    Then SDE \eqref{Eq:SDE} has a unique weak solution $X$. Moreover, for any $\lambda>0$ and $\beta>1$, $X$ satisfies 
    \begin{equation}\label{Eq:Krylov2}
        \begin{aligned}
            &\bE \int_0^\infty \e^{-\lambda t} f(X_t)\,\d t\leq \frac{C}{\lambda} \inf\l\{ \lambda>0: \int_{\mR^d} \l( \exp \l|f(x)/\lambda\r|^\beta -1 \r) \d x\leq 1\r\}, \quad f\in L_A. 
	\end{aligned}  
    \end{equation}
    Here the constant $C$ does not depend on $f$. 
\end{corollary}
\begin{remark}
	\begin{enumerate}[(i)]
		\item When we seek to compare the classical results of Stroock-Varadhan for diffusion processes (see \cite{stroock2007multidimensional}) with the aforementioned findings, a natural question emerges: can we relax the assumption \ref{Aspt:a-2} in Theorem \ref{Thm:Main2} to the point where $x\mapsto a(x,z)$ is only uniformly continuous? We believe the answer is yes, but currently, we do not have a solution to this problem. One obstacle is that the singular integrals may not be bounded in general Orlicz spaces (see Remark \ref{rmk:SIO} below for further discussion); 
		\item The inequalities \eqref{Eq:Krylov1} and \eqref{Eq:Krylov2} are commonly referred to as Krylov-type estimates. For nondegenerate It\^o processes, the dominant term on the right-hand side of the inequality is $C\|f\|_{L^d}/\lambda$, which can be derived from the ABP maximal principle (see \cite{krylov1980controlled}). On the other hand, for the Markov process associated with $\cL$, the ``diffusion" rate of the process is exceedingly sluggish, causing the occupation time $\mE_0 \int_{0}^{1} \1_{B_r}(X_t)\,\d t$ to decay more slowly than $r^s$, for any $s>0$, as $r\to 0$. Consequently, we require a stronger integrability condition for $f$ in this case. 
        \item By leveraging the transformations outlined in Example \ref{ex:sde}, and using Theorem \ref{Thm:Main2} we can address the weak well-posedness of SDE \eqref{Eq:SDE} driven by general subordinate Brownian. However, this approach requires verifying that \( \ell(r) \) is H\"older  continuous near \(r=0\), a condition that is generally non-trivial to validate. In some specific cases, asymptotic analysis of \( \ell(r) \) near the origin offers a practical pathway to address this requirement (see more discussion after Proposition \ref{Prop:j-regular}).
	\end{enumerate}
\end{remark}

\subsection{Approach}\label{sec-approach}
    As mentioned before, by comparing \eqref{eq:stable} and \eqref{Eq:Lu}, one can roughly interpret \eqref{Eq:Lu} as dealing with the ``$\alpha=0$" case, which turns out to be more challenging. Two of the reasons for this are:
    \begin{enumerate}[(1)]
	\item Considering the nearly integrable nature of the kernel $|z|^{-d}\1_{B_1}(z)$ (indeed $z\mapsto |z|^{-d}\1_{B_1}(z)$ is in weak $L^1$ space), it can be intuitively predicted that the regularity improvement of $L^{-1}$ will be rather weak. As a result, it seems difficult to deal with such operator using standard elliptic theory; 
	\item Although when $a=1$, operator $L$ corresponds to a nice L\'evy process $Y$, it is essential to highlight that $Y$ does not exhibit scaling invariance, a distinctive characteristic that sets it apart from $\alpha$-stable processes. The lack of scaling invariance introduces a substantial challenge in the study of these Markov operators, as it renders the conventional scaling techniques, highly effective for $\alpha$ stable-like operators, but ineffective in this context.
    \end{enumerate}

    In this work, we mainly use a modified Littlewood-Paley type decomposition and some tools from probability theory to deal with our two issues. In order to demonstrate the robustness of our approach, in Appendix \ref{Appendix-stable}, we apply the classic Littlewood-Paley theory and scaling techniques to reprove and extend the main result of \cite{bass2009regularity} for stable-like operators. As can be seen from our proof, it is convenient to consider Schauder-type estimates in the H\"older-Zygmund space $\sC^s$. Here $\sC^s$ is defined using $\Delta_j$, which is called the non-homgeneous dyadic blocks (cf. \cite{bahouri2011fourier}). When $s>0$ with $s\notin \mN$, it is well known that $\sC^s$ coincides with the usual H\"older space $C^s$.

\smallskip

    Naturally, we attempt to extend the above mentioned idea to the study of $\cL$, but immediately face difficulties. The usual dyadic decomposition is no longer applicable to the situation we are concerned with. For example, for the operator $\cL=-\log(I-\Delta)$, a naive substitute of $\sC^s$ is the function space $\sX^s$ defined by
\begin{equation}\label{Eq:X}
	\sX^s = \Big\{u\in \sS'(\mR^d): \sup_{j\geq -1} (2+j)^s\|\Delta_j f\|_\infty <\infty \Big\}
\end{equation}
(see \cite{chae2012logarithmically}). 
However, as we demonstrate in Appendix \ref{Appendix-exmple}, $\sX^1$ contains unbounded discontinuous functions. Thus, classical decomposition theory is no longer applicable to the problems we are concerned with.

To overcome the aforementioned obstacle, we leverage the concept introduced in \cite{kassmann2017intrinsic} and propose a novel decomposition for distributions, which we refer to as the ``intrinsic dyadic decomposition" or ``$\psi$-decomposition" (where $\psi$ denotes any positive, increasing function, though we ultimately choose it as the L\'evy exponent of $Z$ for our purposes). This new decomposition replaces the role of $\Delta_j$ with the $\psi$-dyadic block $\Pi_j^{\psi}$, which is defined in Section \ref{sec-IDD}. Additionally, the generalized Orlicz-Besov space $B^{\psi,s}_{A}$ (as presented in Section \ref{sec-orlicz}) supplants the position of the Besov space $B^{s}_{p,\infty}$ in classical theory. One of the key observations in proving Schauder-type estimates is Theorem \ref{Thm:chart}, which serves as an analogy to Theorem \ref{Thm:chart0} and establishes that the generalized H\"older space $C_{\rho}^{s }$ and $B^{\psi,s}_{\infty}$ are consistent.

The second issue at hand pertains to the inadequacy of scaling methods in proving Schauder estimates for the very low-order operators (further details are provided in Appendix \ref{Appendix-stable}). Consequently, we begin by focusing on a specific subset of non-local operators, namely the infinitesimal generators of subordinate Brownian motions with slowly varying symbols. Using the favorable analytical properties of the subordinator $S$ and the Gaussian kernel, we establish the desired a priori estimates in Orlicz-Besov spaces (as stated in Theorem \ref{Thm:Main0}). Subsequently, we gradually expand these estimates to encompass the general case using classical techniques developed for differential operators.

\smallskip

Regarding the martingale problem associated with $\cL$, following \cite{stroock2007multidimensional}, one can see that well-posedness is a consequence of the solvability of \eqref{Eq:PE} in generalized H\"older spaces. However, certain crucial properties of the Markov process corresponding to $\cL$ cannot be inferred from estimates in H\"older-type spaces. For example, such estimates do not inform us whether the occupation time $\int_{0}^T \1_{B}(X_s) \d s$ is nicely upper or lower bounded. To address this, we turn to $L^\infty$-estimates for \eqref{Eq:PE} under the assumption that $f$ belongs to some Orlicz space. The $L^\infty$ bounds for solutions to \eqref{Eq:PE} are derived via a Morrey-type inequality \eqref{Eq:Morrey2} and an a priori estimate for solutions in Orlicz-Besov spaces (as described in Lemma \ref{Le-Main1}). We note that although $L^p$-estimates can also be established (see for instance \cite{kang2025lqlp}), they are inadequate for obtaining the probabilistic results we require since the following embedding result fails for any $p<\infty$: $[I+\log(I -\Delta)]^{-1}L^p \hookrightarrow L^\infty$. 

{\bf Organization}: The paper is organized into several sections. Section \ref{Sec-Pre} provides a review of Orlicz spaces, slowly varying functions, and some important properties of subordinators with slowly varying Laplace exponents. In Section \ref{Sec-OB}, we introduce the $\psi$-decomposition, generalized Orlicz-Besov spaces, and present the proof of Theorem \ref{Thm:chart}. This section also includes crucial Morrey-type inequalities. In Section \ref{Sec-Main}, we present the proof of Theorem \ref{Thm:Main1} and Corollary \ref{Cor:Log}. Section \ref{Sec-MP} investigates the martingale problem associated with $\cL$. Finally, to illustrate the technical difficulties we mentioned above and the need to introduce the new dyadic decomposition, we offer an appendix including a brief discussion on a theorem by Bass \cite{bass2009regularity} and a remark on the space $\sX^s$. 

\smallskip

{\bf Notations}. Reciprocal of $f$ is denoted by $\frac{1}{f}$ and $f^{-1}$ denote the inverse function of $f$. The letter $c$ or $C$ with or without subscripts stands for an unimportant constant, whose value may change in different places. We use $a \asymp b$ to
denote that $a$ and $b$ are comparable up to a constant, and use $a\lesssim b~(a\gtrsim b)$ to denote $a\leq Cb~(a\geq Cb)$ for some constant $C$.

\section{Preliminaries}\label{Sec-Pre}
\subsection{Orlicz space}\label{sec-orlicz}
In this section, we review some basic facts about Orlicz space. Most of the results can be found in \cite[Chapter VIII]{adams2003sobolev}. 

\smallskip

The notion of Orlicz space extends the usual notion of $L^p$ space. The function $t^p$ entering the definition of $L^p$ is replaced by a more general convex function $A(t)$, which is called an $N$-function. 

\begin{definition}\label{Def-Young}
    Let \(\alpha(t)\) be a real valued function defined  on $[0,\infty)$ and having the following properties: 
    \begin{enumerate}[(a)]
	\item $\alpha(0)=0$, $\alpha(t)>0$ if $t>0$, $\lim_{t\to\infty}\alpha(t)=\infty$; 
	\item $\alpha(t)$ is nondecreasing; 
	\item $\alpha(t)$ is right continuous. 
    \end{enumerate}
    Then the real valued function $A$ defined on $ [0,\infty)$ by 
    \[
	A(t)=\int_0^t \alpha(s)\, \d s, \quad t\geq 0
    \]
    is called an $N$-function. 
\end{definition}
Typical examples of $N$-functions include $A(t)=t^p$ for $1<p<\infty$ and $A(t)=\e^{t^\beta}-1$ for $\beta\geq 1$. 

The Legendre transform of a convex function $A$ on $[0,\infty)$ is given by 
\begin{equation}\label{Eq:A*}
	A_{*}(s) :=\sup_{t\geq 0} [st-A(t)]. 
\end{equation}
By this, one can verify that $A_{*}: [0,\infty)\to [0,\infty)$ is also an $N$-function (cf. \cite{adams2003sobolev}). By $A^{-1}$ we denote the inverse of $A$: 
\[
A^{-1}(s)=\inf\l\{t: A(t)> s\r\},\quad s\geq 0. 
\]
The following result can be found in \cite[page 265, equation (4)]{adams2003sobolev}. 
\begin{proposition}\label{Prop-AA*}
	Assume that $A: [0,\infty)\to [0,\infty)$ is an $N$-function, then 
	\begin{equation}\label{Eq:AA*}
		s\leq A_{*}^{-1}(s) \cdot A^{-1}(s)\leq 2 s. 
	\end{equation}
\end{proposition}

Given an $N$-function $A$, for each measurable function $f$ on $\mR^d$, we define 
\[
I_A(f):= \int_{\mR^d} A(|f|(x))\,\d x. 
\]

\begin{definition}[Luxemburg norm]\label{Def-Luxemburg}
    Let $A$ be an $N$-function. For any measurable function $f$, define
    \[
	\|f\|_A:=\inf \left\{\lambda>0: I_A(f/\lambda) \leq 1\right\} \in[0, \infty].
    \]
    The Orlicz space $L_A$ is the collection of all functions with finite Luxemburg norm. \(L_A\) is a Banach space. 
    In general, the space \(L_A\) is not necessarily separable unless \(A\) satisfies the \(\Delta_2\)-condition (see \cite[8.6, page 266]{adams2003sobolev}). However, one can consider a subspace of \(L_A\), denoted by \(E_A\), which is defined as the closure in \(L_A\) of \(L_A\cap C_0\), where \(C_0\) denotes the space of continuous functions vanishing at infinity. 
    
\end{definition} 
Orlicz spaces have good duality and interpolation properties, making them useful in many areas of mathematics, including probability theory, functional analysis, and PDEs. The following two propositions can also be found in \cite[8.11, page 269 and Lemma 8.17, page 272]{adams2003sobolev}. 
\begin{proposition}
    For all $f\in L_A$ and $g\in L_{A^*}$, it holds that \begin{equation}\label{Eq:Holder}
        \int |fg| \leq 2 \|f\|_{A} \|g\|_{A^*}. 
    \end{equation}
\end{proposition}
\begin{proposition}
    \[
    \|f\|_{A} \asymp \sup \left\{\|fg\|_{1}: \int_{\mR^d} A_{*}(g(x))\,\d x \leq 1\right\}.
    \]
\end{proposition}

The following Young's inequality is crucial. 
\begin{proposition}[{\cite[Theorem 2.5]{oneil1965fractional}}]
	Assume that $A, B$ and $C$ are $N$-functions and that $A^{-1}(t)B^{-1}(t)\leq t C^{-1}(t)~(\forall t\geq 0)$, then
	\begin{equation}\label{Eq:Young}
		\|f*g\|_C \leq 2 \|f\|_A \|g\|_B.  
	\end{equation}
\end{proposition}

\subsection{Regular variation}\label{Sec-SVF}
\begin{definition}\label{Def-SVF}
	A measurable function $f: (0,\infty)\to (0,\infty)$ is regularly varying at infinity with index $\alpha\in \mR$ (denoted by $f\in \cR_\alpha(\infty)$), if for all $b>0$, 
	\[
	\lim_{x\to\infty}\frac{f(b x)}{f(x)}=b^\alpha. 
	\]
	If $\alpha=0$, then we say that $f$ is slowly varying. Regular (slow) variation at zero can be
	defined similarly, by changing $x\to\infty$ to $x\to 0^+$. 
\end{definition}

The following two Karamata's theorems (cf. \cite[Theorem 1.5.11]{bingham1987regular}) will be used frequently. 
\begin{lemma}[Karamata's representation theorem]\label{Le-Rep-Karamata}
    A function $f$ is slowly varying if and only if there exists $M>0$ such that for all $x \geq M$ the function can be written in the form
    $$
    f(x)=c(x)\exp \left(\int_M^x \frac{\varepsilon(t)}{t} d t\right)
    $$
    where $c(x)$ is a bounded measurable function of a real variable converging to a finite number as $x$ goes to infinity, and $\varepsilon(x)$ is a bounded measurable function of a real variable converging to zero as $x$ goes to infinity.
\end{lemma}
\begin{lemma}[Karamata's integral theorem]\label{Le-Inter-Karamata}
	Let $f$ be a slowly varying function at infinity. Then  
	\begin{enumerate}[(i)]
		\item for any $\sigma\geq -1$ and $M>0$,  \[\lim_{s\to\infty}\frac{s^{\sigma+1} f(s)}{\int_M^s t^\sigma f(t) \,\d t} = \sigma+1; 
		\] 
		\item for any $\sigma<-1$, 
		\[
		\lim_{s\to\infty} \frac{s^{\sigma+1} f(s)}{\int_s^{\infty} t^\sigma f(t) \,\d t}=-(\sigma+1). 
		\]
	\end{enumerate}
\end{lemma}

\subsection{Subordinators and jump kernels}\label{subsec:subordinator}
Recall that $S$ is a subordinator without drift. $\phi$, the Laplace exponent of $S$ is given by \eqref{eq:rept-phi}. Since $(s t)\e^{-s t}\leq 1-\e^{-s t}$ for each $s, t>0$, by \eqref{eq:rept-phi}, we have 
\begin{equation}\label{eq:phi'-phi}
	 s\phi'(s)\leq \phi(s). 
\end{equation}
Moreover,  \eqref{eq:rept-phi} also implies
\begin{equation}\label{eq:phi-Pi}
	 \phi(u^{-1})\geq u^{-1} \int_0^u \e^{-t/u}\Pi((t,\infty)) \d t \geq \e^{-1} \Pi((u,\infty)), \quad u>0. 
\end{equation}
Recall also that $B$ is a $d$-dimensional Brownian motion independent of $S$, $Z_t= \sqrt{2}B_{S_t}$, and that 
\[
\psi(\xi)= -\log \bE \e^{i \xi\cdot Z_1}=\phi(|\xi|^2)
\]
is the L\'evy exponent of the subordinate Brownian motion $Z$.  $J(z)$, the jump kernel of $Z$, is given by \eqref{Eq:j}. 

For any $\lambda\geq 0$, let 
\[
U_\lambda(\d u)= \int_0^\infty \e^{-\lambda t}~\bP(S_t\in \d u)\,\d t
\]
be the $\lambda$-potential measure of $S$. By Fubini's theorem,  
\begin{align*}
	 \int_{0}^{\infty} \e^{-\tau u} U_\lambda (\d u)=&  \int_{0}^{\infty} \e^{-\tau u} \int_0^\infty \e^{-\lambda t}\bP(S_t\in \d u)~\d t\\
	=&  \int_{0}^{\infty} \e^{-\lambda t}~\d t\int_0^\infty \e^{-\tau u}  \bP(S_t\in \d u)\\
	=& \int_0^\infty \e^{-(\lambda+\phi(\tau))t}~\d t
	= \frac{1}{\lambda+\phi(\tau)}. 
\end{align*}

This implies 
\begin{equation}\label{eq:phi-U}
	 \frac{1}{\lambda+\phi(\tau^{-1})}\geq \int_0^{\tau} \e^{-\frac{u}{\tau}}  U_{\lambda}(\d u) \geq \e^{-1} U_{\lambda}((0,\tau)), \quad \tau>0. 
\end{equation}

 The following proposition will be used in the proof for our main results. 
\begin{proposition}\label{Prop:j-regular}
    Suppose $\phi$ satisfies \ref{Aspt:J-1}, then for any \(r\ll 1\), it holds that
    \begin{equation}\label{eq:int-zJ(z)}
        \int_{B_r} |z|J(z)\,\d z \lesssim r\psi(r^{-1}), 
    \end{equation}
    and \begin{equation}\label{eq:int-J(z)}
        \int_{B_r^c} J(z)\,\d z \lesssim \psi(r^{-1}).
    \end{equation}
\end{proposition}
\begin{proof}
    For \eqref{eq:int-zJ(z)}, in virtue of Fubini's theorem,\eqref{Eq:j} and  \eqref{eq:phi-Pi},  we have 
    \begin{equation*}
	\begin{aligned}
            \int_{B_r} |z|J(z)\,\d z\asymp& \int_0^r j(s)s^d\,\d s \lesssim \int_0^r s^{d} \d s \int_0^\infty t^{-\frac{d}{2}} \e^{-\frac{s^2}{4t}} \Pi(\d t)\\
            = &  \int_0^\infty t^{-\frac{d}{2}} \Pi(\d t) \int_0^r s^{d} \e^{-\frac{s^2}{4t}} \d s \asymp \int_0^\infty  t^{\frac{1}{2}} \Pi(\d t) \int_0^{\frac{r^2}{4t}} u^{\frac{d-1}{2}} \e^{-u} \d u\\ 
            \asymp& \int_0^{r^2}  t^{\frac{1}{2}} \Pi(\d t) \int_0^{\frac{r^2}{4t}} u^{\frac{d-1}{2}} \e^{-u} \d u + \int_{r^2}^\infty  t^{\frac{1}{2}} \Pi(\d t) \int_0^{\frac{r^2}{4t}} u^{\frac{d-1}{2}} \e^{-u} \d u \\
            \asymp& \int_0^{r^2}  t^{\frac{1}{2}} \Pi(\d t) + r^{d+1} \int_{r^2}^\infty  t^{-\frac{d}{2}} \Pi(\d t)\\
            \lesssim & \int_0^{\infty} \!\!\! \int_0^\infty \1_{\{ s<t\leq r^2 \}} s^{-\frac{1}{2}}  \Pi(\d t)\, \d s+ r \Pi((r^2, \infty))\\
            \lesssim & \int_0^{r^2} s^{-\frac{1}{2}}  \Pi((s,\infty)) \d s + r \Pi((r^2, \infty)) \\
            \overset{\eqref{eq:phi-Pi}}{\lesssim} & \int_0^{r^2} s^{-\frac{1}{2}}\phi(s^{-1}) \d s + r\phi(r^{-2})\lesssim r\psi(r^{-1}), \quad 0<r\ll 1. 
	\end{aligned}
    \end{equation*}
    Here we used Lemma \ref{Le-Inter-Karamata} (ii) with $\sigma=-\frac{3}{2}$ and $s = r^{-2}$ in the last inequality. 

    For \eqref{eq:int-J(z)}, 
    again by Fubini's theorem, \eqref{eq:phi-Pi} and Lemma \ref{Le-Inter-Karamata} (ii) with $\sigma=-2$ and $s = r^{-2}$, we have 
    \begin{equation*}
	\begin{aligned}
            \int_{B_r^c} J(z)\,\d z \asymp& \int_r^\infty s^{d-1} j(s)  \d s = \int_r^\infty s^{d-1} \d s \int_0^\infty t^{-\frac{d}{2}} \e^{-\frac{s^2}{4t}} \Pi(\d t)\\
            = & \int_0^\infty t^{-\frac{d}{2}} \Pi(\d t) \int_r^\infty s^{d-1} \e^{-\frac{s^2}{4t}} \d s \asymp \int_0^\infty \Pi(\d t) \int_{\frac{r^2}{4t}}^\infty u^{\frac{d}{2}-1} \e^{-u} \d u \\ 
            = & \int_0^\infty u^{\frac{d}{2}-1} \e^{-u}  \d u \int_{\frac{r^2}{4u}}^\infty \Pi(\d t) \overset{\eqref{eq:phi-Pi}}{\lesssim} \int_0^\infty u^{\frac{d}{2}-1} \e^{-u}  \phi(4u/r^2)  ~ \d u \\
            \lesssim & \phi(r^{-2}) \int_0^{\frac{1}{4}}  u^{\frac{d}{2}-1} \e^{-u} \d u + \int_{\frac{1}{4}}^\infty u^{-2} \phi(4u/r^2) \d u\\
            \lesssim & \phi(r^{-2})+r^{-2} \int_{r^{-2}}^\infty  s^{-2} \phi(s)\d s\lesssim \psi(r^{-1}). 
	\end{aligned}
    \end{equation*}
\end{proof}

To verify that Example \ref{ex:sde} and Example \ref{ex:logLaplace} satisfy our assumptions, we need to conduct a more detailed analysis. The specific calculations are as follows.

For the gamma subordinator, recall that \(\Pi(\d t)= \e^{-t}t^{-1} \d t\). By \eqref{Eq:j}, 
\[
        j(r)=(4\pi)^{-\frac{d}{2}} \int_0^\infty t^{-\frac{d}{2}-1} \e^{-\frac{r^2}{4t}-t} \d t=\frac{1}{r^d\ell(r)},  
\]
which yields that 
\[
    1/\ell(r)= (4\pi)^{-\frac{d}{2}}r^d \int_0^\infty t^{-\frac{d}{2}-1} \e^{-\frac{r^2}{4t}-t} \d t
\]
and 
\[
    (1/\ell(r))'=(4\pi)^{-\frac{d}{2}}d~r^{d-1}\int_0^\infty t^{-\frac{d}{2}-1} \e^{-\frac{r^2}{4t}-t} \d t -(4\pi)^{-\frac{d}{2}}\frac{r^{d+1}}{2} \int_0^\infty t^{-\frac{d}{2}-2} \e^{-\frac{r^2}{4t}-t} \d t
\]
For any \(\gamma\in \mR\), set 
\[
    I_\gamma(r) := \int_0^\infty t^{\gamma} \e^{-\frac{r^2}{4t}-t} \d t.  
\]
By integration by parts formula, we get 
\begin{equation*}
    \begin{aligned}
        I_{\gamma}(r) =& \frac{1}{\gamma+1} \int_0^\infty t^{\gamma+1} \e^{-\frac{r^2}{4t}-t} (1-r^2/(4t^2)) \d t \\
        =&\frac{I_{\gamma+1}(r)}{\gamma+1} -\frac{r^2}{4(\gamma+1)} I_{\gamma-1}(r), \quad \gamma\neq -1. 
    \end{aligned}
\end{equation*}
Therefore, 
\[
1/\ell(r)=2^{-d}\pi^{-\frac{d}{2}} r^d I_{-\frac{d}{2}-1}(r)
\]
and 
\begin{align*}
    (1/\ell(r))' = (4\pi)^{-\frac{d}{2}}  d r^{d-1} \l(I_{-\frac{d}{2}-1}(r)-\frac{r^2}{2d} I_{-\frac{d}{2}-2}(r)\r) = -2^{1-d}\pi^{-\frac{d}{2}} r^{d-1} I_{-\frac{d}{2}}(r). 
\end{align*}
Noting that 
\begin{equation*}
    \begin{aligned}
    I_{\gamma}(r)=&\int_0^\infty t^{\gamma} \e^{-\frac{r^2}{4t}-t} \d t = \int_0^{r^2/4} t^{\gamma}  \e^{-\frac{r^2}{4t}-t} \d t + \int_{r^2/4}^\infty t^{\gamma}  \e^{-\frac{r^2}{4t}-t} \d t \\
    \asymp &  
    \begin{cases}
    r^{2\gamma+2}, \quad \gamma<-1\\ 
    -\log (r), \quad  \gamma=-1\\
    1, \quad \gamma>-1, 
    \end{cases} ~ 0<r\ll 1.  
    \end{aligned}
\end{equation*}
and \(I_{\gamma}(r)\lesssim \e^{-r/3},~ r\gg 1 \), we obtain  
\[
    j(r)=1/(r^d\ell(r)) \asymp I_{-\frac{d}{2}-1}(r) \asymp r^{-d}, ~ r\to 0^+  
\]
and 
\[
    j(r)=1/(r^d\ell(r))\lesssim r^{-d} \e^{-r/3}, ~r>0. 
\]
Moreover, we have 
\[
      |\ell'(r)|\lesssim r^{d-1} I_{-\frac{d}{2}}(r)\lesssim 1, ~ r>0, 
\]
which implies that the coefficient \(a\) in \eqref{Eq:L-sde} satisfies \ref{Aspt:a-1} and \ref{Aspt:a-2}, provided that \(\sigma\) is \(\rho^\alpha\)-continuous and \(c_0 |\xi| \leq |\sigma \xi|\leq c_0^{-1}|\xi|\) for all \(\xi\in\mR^d\). 

\section{Intrinsic dyadic decomposition, function spaces and embedding theorems}\label{Sec-OB}

The contents of this section are purely analytical. Generalized Orlicz-Besov spaces, which is defined by a refined Littlewood-Paley decomposition called $\psi$-decomposition, are introduced. Such function space is a natural extension of Besov space $B^s_{p,\infty}$ (see \cite{bahouri2011fourier} for its definition). Crucial Morrey-type inequalities are also proved in this section. 

\subsection{Intrinsic dyadic decomposition}\label{sec-IDD}
Let $\psi: (0,\infty)\to (0,\infty)$ be a strictly increasing function such that $\lim_{R\to \infty}\psi(R)=\infty$. We point out that in this section, $\psi$ needs not be the L\'evy exponent of a  subordinate Brownian motion. Set 
\[\psi^{-1}(r):=\inf \l\{s>0: \psi(s)>r\r\}, ~ ~r>0. 
\]
Let $\chi$ be a rotationally symmetric, nonnegative and smooth function with compact support such that 
\[
\chi(\xi)= \chi(|\xi|)=
\begin{cases} 1  \quad & \hbox{when }  |\xi|\leq 3/4, \\
	0  &\hbox{when }  |\xi|\geq 1.
\end{cases} 
\]
Let $\varrho: = \cF ^{-1}(\chi)$. We define   
\[\bar{\chi}(\xi)= \chi(2\xi)  \ \mbox{ and }\ \widetilde {\chi}(\xi)= \chi(\xi/2); \quad   \bar{\varrho}=\cF ^{-1}(\bar{\chi})\  \mbox{ and } \  \widetilde{\varrho}=\cF ^{-1}(\widetilde{\chi}). 
\] 
Given $j\geq -1$, set 
\[
\chi_j(\xi):= \chi\l(\frac{\xi}{ \psi^{-1}(2^j)}\r),  \quad \varrho_j(x):= \cF ^{-1}(\chi_j)(x)= \l( \psi^{-1}(2^j)\r)^d \varrho \l( \psi^{-1}(2^j) x\r).
\]
Similarly, we define 
\[
\bar{\chi_j} = \bar{\chi}\l(\frac{\xi}{ \psi^{-1}(2^j)}\r)=\chi\l(\frac{2\xi}{ \psi^{-1}(2^j)}\r), 
\quad \bar{\varrho_j}(x)= \l( \tfrac{1}{2}\psi^{-1}(2^j)\r)^d \varrho \l(\tfrac{1}{2}\psi^{-1}(2^j) x\r), 
\]
and
\[
\widetilde{\chi}_j = \widetilde{\chi}\l(\frac{\xi}{ \psi^{-1}(2^j)}\r)=\chi\l(\frac{\xi}{ 2\psi^{-1}(2^j)}\r), 
\quad \widetilde{\varrho_j}(x)= \l( 2 \psi^{-1}(2^j)\r)^d \varrho \l(2\psi^{-1}(2^j) x\r).  
\]
The $\psi$-dyadic block $\Pi^{\psi}_j$ is formulated by 
\begin{equation*}
	\Pi^{\psi}_j f:= \cF ^{-1} ((\chi_{j+1}- \chi_{j})\cF (f))=(\varrho_{j+1}-\varrho_j)*f, \ j\geq 0   
\end{equation*}
and 
\begin{equation*}
	 \Pi^{\psi}_{-1} f:= \cF ^{-1} \l(\chi\l(\frac{\xi}{\psi^{-1}(1)}\r) \cF (f)(\xi)\r)=\varrho_0*f.
\end{equation*}
Similarly, for each $j\geq 0$, we define 
\[
\widetilde{\Pi}^{\psi}_j f:= \cF ^{-1} ((\widetilde{\chi}_{j+1}- \bar{\chi}_{j})\cF (f)). 
\]
Obviously, for each $j\geq 0$, it holds that 
\[
\widetilde{\Pi}^{\psi}_j \Pi^{\psi}_j = \Pi^{\psi}_j \widetilde{\Pi}^{\psi}_j= \Pi^{\psi}_j.  
\]

The following simple result will be used several times later. 
\begin{lemma}
	Let $A$ be an $N$-function or $A=\infty$. It holds that 
	\begin{equation}\label{Eq:Bernstein}
		\|\nabla \Pi_{j}^{\psi} f\|_A \leq C \psi^{-1}(2^{j+1})~\|\Pi_{j}^{\psi} f\|_A,\quad j\geq -1. 
	\end{equation}
\end{lemma}
\begin{proof}
	Noting that $\Pi_{j}^{\psi} f = \widetilde{\varrho}_{j+1}* (\Pi_{j}^{\psi} f)$, by Young's inequality, we get 
	\begin{align*}
		\begin{aligned}
			\|\nabla \Pi_{j}^{\psi} f\|_A\leq& \|\nabla \widetilde{\varrho}_{j+1}* (\Pi_{j}^{\psi} f)\|_A \overset{\eqref{Eq:Young}}{\leq} \|\nabla \widetilde{\varrho}_{j+1}\|_1 \|\Pi_{j}^{\psi} f\|_{A}\\
			\leq& C \psi^{-1}(2^{j+1})~\|\Pi_{j}^{\psi} f\|_A. 
		\end{aligned}
	\end{align*}
	
\end{proof}

\subsection{Generalized Orlicz-Besov and Orlicz-Bessel potential spaces}
\begin{definition}
	Let $s\in \mR$ and $A$ be an $N$-function. The $\|\cdot\|_{{\psi, s}; {A}}$-norm of a distribution $f$ is given by 
	\begin{equation*}
		\|f\|_{\psi,s;A}:= \sup_{j\geq -1} 2^{js} \|\Pi^{\psi}_j f\|_A, 
	\end{equation*}
	and the collection of all distributions with finite $\|\cdot\|_{\psi,s;A}$-norm is denoted by $B^{\psi,s}_{A}$. If $A=\infty$, we define the space $B^{\psi,s}_{\infty}$ in the same way, and also denote it by $\sC_{\psi}^s$. For simplicity, we denote $B^{\psi, 1}_{A}~(\sC_{\psi}^1)$ as $B^{\psi}_A~ (\sC_\psi)$. 
\end{definition}
By the definition of $B^{\psi,s}_A$, it is easy to get the following interpolation result: 
\begin{lemma}\label{Le-Inter}
	Suppose $s, s_0, s_1\in \mR$, $\theta\in (0,1)$ and $s=\theta s_0+(1-\theta)s_1$. Then 
	\begin{equation*}
		\|f\|_{B^{\psi, s}_A}\leq \|f\|_{B^{\psi,s_0}_A}^{\theta} \|f\|_{B^{\psi,s_1}_A}^{1-\theta}.  
	\end{equation*}
\end{lemma}

We also introduce the definition of Orlicz-Bessel potential spaces. Recall that $E_A$ is defined in Definition \ref{Def-Luxemburg}.
\begin{definition} \label{def:obp}
	Let $A$ be an $N$-function. The space 
	\begin{equation*}
		H^{\psi}_A:= \l\{f\in \sS'(\mR^d): \cF^{-1}[(1+\psi(\cdot)) \cF(f)]\in E_A\r\}
	\end{equation*}
	with the norm 
	\begin{equation*}
		\|f\|_{H^\psi_A}:= \|(1+\psi(\sqrt{-\Delta})) f\|_{A} 
	\end{equation*}
	is called generalized Orlicz-Bessel potential space. 
\end{definition}

We remark that in Definition \ref{def:obp} we use the restricted space $E_A$ instead of $L_A$ for the following reasons. In fact, we will use the space $H^{\psi}_A$ in the proof of Theorem \ref{Thm:Morrey1}, and we first prove Theorem \ref{Thm:Morrey1} for `smooth' functions (which are dense in $E_A$, but may not be dense in $L_A$) and then for general functions by the standard approximation arguments. On the other hand, $E_A$ already contains enough `good' functions in many cases (for instance in the case of the hypothesis of Corollary \ref{cor:morry-bp}). 

\begin{definition}\label{Def-Holder}
	Let $s>0$. Let $\omega: (0,\infty)\to (0,\infty)$ be a strictly increasing function satisfying $\lim_{r \downarrow 0} \omega(r)=0$. Assume $\omega^s(t)/t\to \infty$ as $t\to0$. The $\|\cdot\|_{C_{\omega}^s}$-norm of a measurable function $f$ is given by 
	\[
	\|f\|_{C_{\omega}^s}:= \|f\|_\infty+\sup_{x\neq y} \frac{|f(x)-f(y)|}{\omega^s(|x-y|)}.  
	\]
	The collection of all functions with finite $\|\cdot\|_{C_{\omega}^{s}}$-norm is denoted by $C_{\omega}^s$, which is referred to as the generalized H\"older space. 
\end{definition}
Recall that 
\[
\rho(r)=\rho_\psi(r):= \frac{1}{\psi(r^{-1})}.
\]
The following result gives a characterization of generalized H\"older space $C_{\rho}^{s}$ in terms of $\psi$-decomposition. 
\begin{theorem}\label{Thm:chart}
	Let $s>0$. Let $\rho$ be the function defined by \eqref{Eq:rho}. Assume 
	\begin{equation}\label{aspt-int-psi}
		r\mapsto \frac{r}{\psi^s(r)} \ \mbox{ is an increasing function and }~ \int_{\eps}^R \frac{\d r}{\psi^s(r)}\lesssim_\eps \frac{R}{\psi^s(R)} \mbox{ for all } \eps\ll 1 \ll R. 
	\end{equation}
	Then there is a constant $c\in (0,1)$ independent of $f$ such that 
	\begin{equation}
		c \|f\|_{\sC_{\psi}^s} 
		\leq \|f\|_{C_{\rho}^{s}}\leq c^{-1} \|f\|_{\sC_{\psi}^s}. 
	\end{equation}
\end{theorem}
\begin{proof}
	We only prove the case that $s=1$ (the proof for the general case is similar). Assume that $\|f\|_{\sC_{\psi}}<\infty$. For any $|x-y|\ll1$,  
	\begin{align*}
		|f(x)-f(y)| \leq &\sum_{j\geq -1} |\Pi^{\psi}_j f(x)-\Pi^{\psi}_j f(y)| \\
		\lesssim &\sum_{j\geq -1}  \|\Pi^{\psi}_{j} f\|_\infty \wedge |x-y|\|\nabla\Pi^{\psi}_jf\|_\infty. 
	\end{align*}
	Thanks to \eqref{Eq:Bernstein}, for each $j\geq -1$, 
	\[
	\|\nabla\Pi^{\psi}_jf\|_\infty \lesssim \psi^{-1}(2^{j+1}) \|\Pi^{\psi}_j f\|_\infty.  
	\]
	So for each $K\in \mN$, it holds that  
	\begin{equation}\label{eq:f-dif}
		\begin{aligned}
			|f(x)-f(y)| \lesssim &\|f\|_{\sC_{\psi}} \sum_{j\geq -1} 2^{-j}\wedge |x-y|  \psi^{-1}(2^{j+1}) 2^{-j}\\
			\lesssim &\|f\|_{\sC_{\psi}} \sum_{K+1}^\infty 2^{-j}+\|f\|_{\sC_{\psi}} |x-y| \sum_{j=-1}^K  \psi^{-1}(2^{j+1}) 2^{-j} \\
			\lesssim & \|f\|_{\sC_{\psi}} \l(2^{-K}+ |x-y|\int_0^{K+1}  \psi^{-1}(2^t) 2^{-t}\,\d t \r). 
		\end{aligned}
	\end{equation}
	
	For $R\gg 1$, by integration by parts and \eqref{aspt-int-psi}, we have  
	\begin{align*}
		\int_0^R  \psi^{-1}(2^t) 2^{-t} \d t\lesssim& \int_1^{2^R}\psi^{-1}(r) r^{-2} \d r =-\int_{\psi^{-1}(1)}^{\psi^{-1}(2^R)} u\,\d \l(\frac{1}{\psi(u)}\r)\\
		\lesssim&\frac{u}{\psi(u)}\Big|^{\psi^{-1}(1)}_{\psi^{-1}(R)}+\int_{\psi^{-1}(1)}^{\psi^{-1}(2^R)}  \frac{\d u}{\psi(u)}\lesssim \psi^{-1}(2^R) 2^{-R}. 
	\end{align*}
	Choosing $R=K+1=[\log_2 \psi(|x-y|^{-1})]$, one sees that 
	\begin{equation}\label{eq:f-dif1}
		2^{-K}+ |x-y|\int_0^{K+1}  \psi^{-1}(2^t) 2^{-t} \d t \lesssim  2^{-K} \lesssim \frac{1}{\psi(|x-y|^{-1})}. 
	\end{equation}
	Plugging \eqref{eq:f-dif1} into \eqref{eq:f-dif}, we see that for any $|x-y|\ll 1$, 
	\[
	\frac{|f(x)-f(y)|}{\rho(|x-y|)}=|f(x)-f(y)|\psi(|x-y|^{-1}) \lesssim \|f\|_{\sC_{\psi}}.
	\]
	This together with the fact that $\|f\|_\infty\lesssim \|f\|_{\sC_{\psi}}$ implies 
	\[
	\|f\|_{C^1_\rho}=\|f\|_\infty+\sup_{x\neq y} \frac{|f(x)-f(y)|}{\rho(|x-y|)}\lesssim \|f\|_{\sC_{\psi}}.  
	\]
	
	On the other hand, suppose that 
	\[
	M=\sup_{x\neq y} \frac{|f(x)-f(y)|}{\rho (|x-y|) } + \|f\|_\infty <\infty. 
	\]
	By the definition of $\Pi^{\psi}_j$ and the fact that $\int\varrho_j=\chi_j(0)=1\, (j\geq 0)$, we have 
	\begin{equation}\label{eq:pif}
		\begin{aligned}
			&|\Pi^{\psi}_j f(x)|= \l| \int_{\mR^d} (\varrho_{j+1}-\varrho_j)(z) f(x+z)\,\d z \r|\\
			=& \l| \int_{\mR^d} (\varrho_{j+1}-\varrho_j)(z) (f(x+z)-f(x)) \,\d z \r|\\
			\leq& \int_{\mR^d} |\varrho(z)| \l[\l| f\l(x+\frac{z}{ \psi^{-1}(2^j)}\r)-f(x) \r| + \l| f\l(x+\frac{z}{ \psi^{-1}(2^{j+1})}\r)-f(x) \r|\r] \,\d z \\
			\leq& M \int_{\mR^d} |\varrho(z)| \l[ 1\wedge \frac{1}{\psi\l(\frac{\psi^{-1}(2^j)}{|z|}\r)} \r]\,\d z + M  \int_{\mR^d} |\varrho(z)| \l[ 1\wedge \frac{1}{\psi\l(\frac{\psi^{-1}(2^{j+1})}{|z|}\r)} \r]\,\d z\\
			\lesssim & M\int_0^\infty  \frac{\varrho(r) r^{d-1}}{1\vee \psi \l(\frac{ \psi^{-1}(2^j)}{r}\r)}\,\d r + M \int_0^\infty \frac{\varrho(r)  r^{d-1}}{1\vee \psi \l(\frac{ \psi^{-1}(2^{j+1})}{r}\r)} \,\d r \quad (j\geq 0). 
		\end{aligned}
	\end{equation}
	So we need to dominate the integration $\int_0^\infty \frac{\varrho(r)  r^{d-1}}{1\vee \psi \l(\frac{s}{r}\r)} \,\d r$. Using the fact that $|\varrho(t)|\lesssim_k (1\wedge t^{-k})\ (\forall k>0)$, one sees that for each $s>0$, 
	\begin{equation*}
		\begin{aligned}
			\int_0^\infty \frac{\varrho(r)r^{d-1}}{1\vee \psi (s/r)}\,\d r =&~ s^d \int_0^\infty \frac{\varrho(s/t)}{t^{d+1} [1\vee \psi (t)]}\,\d t \\
			\lesssim&~ s^{d-k}\int_0^s \frac{t^k}{t^{d+1} [1\vee \psi (t)]}\,\d t  +s^d \int_s^\infty \frac{1}{t^{d+1} [1\vee \psi (t)]}\,\d t\\
			=& I_1+I_2. 
		\end{aligned}
	\end{equation*}
	Choosing $k=d+2$ and using \eqref{aspt-int-psi}, we get  
	\[
	I_1=s^{-2}\int_0^s \frac{t}{ 1\vee \psi (t)} \,\d t  \lesssim s^{-1}\int_0^s \frac{\d t}{1\vee \psi(t)} \overset{\eqref{aspt-int-psi}}{\lesssim} \frac{1}{ 1\vee \psi (s)}. 
	\]
	Moreover,  
	\[
	I_2=s^d \int_s^\infty \frac{1}{t^{d+1} [1\vee \psi (t)]} \d t\leq \frac{1}{1\vee \psi (s)}s^d\int_s^\infty t^{-d-1} \d t\lesssim \frac{1}{1\vee \psi (s)}. 
	\]
	Thus, for each $s>0 $, 
	\begin{equation}\label{eq:pif1}
		\int_0^\infty \frac{\varrho(r)r^{d-1}}{1\vee \psi (s/r)}\,\d r \lesssim \frac{1}{ 1\vee \psi (s)}. 
	\end{equation}
	Combining \eqref{eq:pif} and \eqref{eq:pif1}, we arrive 
	\[
	\|\Pi^{\psi}_j f\|_\infty\lesssim M2^{-j}, \ \mbox{ for all } j\geq -1. 
	\]
	So we complete our proof. 
\end{proof}

\subsection{Morrey-type inequalities}
In this section, we prove two Morrey-type inequalities, \eqref{Eq:Morrey1} and \eqref{Eq:Morrey2}. The novelty of them is that both regularity and integrability of $u$ are characterized by monotone functions $\psi$ and $A$, respectively, and $\psi$ can be a slowly varying function. 

We need to make some necessary preparations. Let $S$ be a subordinator with Laplace exponent $\phi$.  Recall $Z$ is the subordinate Brownian motion corresponding to $S$. Here and below, the function $\psi$ is given by the L\'evy exponent of $Z$, i.e.,
\[
\psi(\xi)= -\log \bE \e^{i \xi\cdot Z_1}=\phi(|\xi|^2).  
\]

In this section, we assume $S$ satisfies 
\begin{enumerate}
	\item[\mylabel{Aspt-phi'-1}{$\bf (\Phi'_{1})$}]: $\phi\in \cR_{\frac{\alpha}{2}}(\infty)$ with $\alpha\in [0,1)$, i.e. $\psi\in \cR_{\alpha}(\infty)$ with $\alpha\in [0,1)$. 
\end{enumerate}

\begin{lemma}
    Assume that $\phi$ satisfies \ref{Aspt-phi'-1}. Then for any \(R\geq 1\), it holds that
    \begin{equation}\label{eq:phi-int}
        \int_1^R \frac{\d r}{\psi(r)}\lesssim \frac{R}{\psi(R)} 
    \end{equation}
    and 
    \begin{equation}\label{eq:t-leq-R}
        \frac{r}{\psi(r)}\lesssim \frac{R}{\psi(R)}, \quad r\in [1,R]. 
    \end{equation}
\end{lemma}
\begin{proof}
    If $\psi\in \cR_{\alpha}(\infty)$ with $\alpha\in [0,1)$, then there exists a function $\ell\in \cR_0(\infty)$ such that $\psi(r)=r^{\alpha}\ell(r)$. For $R\geq 1$, using Lemma \ref{Le-Inter-Karamata} (i), we have  
    \begin{equation*}
    \int_1^R \frac{\d r}{\psi(r)}\lesssim \int_1^R \frac{\d r}{r^\alpha \ell(r)}\lesssim \frac{R^{1-\alpha}}{\ell(R)} \lesssim \frac{R}{\psi(R)}.
    \end{equation*}

    For \eqref{eq:t-leq-R}. By Lemma \ref{Le-Rep-Karamata}, \(\psi(r)=r^\alpha c(r) \exp(\int_1^{r}\eps(t)/t\d t)\), with \(c(r)\to c\in \mR\) and \(\eps(r)\to 0\) as \(r\to\infty\). Therefore, 
    \[
        \frac{r\psi(R)}{R\psi(r)}\lesssim \exp\l(\int_r^{R}\frac{\eps(t)- (1-\alpha)}{t}\d t \r)\lesssim 1, 
    \]
    which implies our desired assertion. 
\end{proof}

\subsubsection{Morrey-type inequality for  Orlicz-Bessel potential spaces}

The following Morrey-type inequality for Orlicz-Sobolev space $W^1_A$ can be found in \cite{adams2003sobolev}.
\begin{proposition}\label{Prop:morrey}
    Let $A$ be an $N$-function. Assume that $\int_1^{\infty}  \frac{A^{-1}(s^d) }{s^2}~ \d t<\infty$, then 
    \[
	\|f\|_{C_{\rho}}\lesssim \|f\|_{A}+\|\nabla f\|_{A}, 
    \]
    where  $W^1_A$ is the Orlicz-Sobolev space and 
    \[
	\rho(r)= \l(\int_{\frac{1}{r}}^\infty \frac{A^{-1}(s^d)}{s^2}\,\d s\r)^{-1}. 
    \]
\end{proposition}
\begin{remark}
	As far as the authors are aware, there is no result guaranteeing that the Riesz transform is bounded in the Orlicz space $L_A$, so in general $H^1_A$ does not coincide with $W^1_A$. Therefore,  $H^1_{A}\hookrightarrow C_\rho$ cannot be derived from Proposition \ref{Prop:morrey}. 
\end{remark}

The following result is about Morrey inequalities for generalized Orlicz-Bessel potential spaces. 
\begin{theorem}\label{Thm:Morrey1}
    Assume that $\phi$ satisfies \ref{Aspt-phi'-1}. Let $A$ be an $N$-function satisfying 
    \[
        \int_1^{\infty}\frac{\d A^{-1}(t^d)}{\psi(t)} ~ \d t<\infty.
    \]
    Let \begin{equation}\label{eq:Psi}
        \Psi(R) := \l( \int_R^\infty \frac{\d A^{-1}(t^d)}{\psi(t)}\,\d t \r)^{-1}.   
    \end{equation}
    Then it holds that
    \begin{equation}\label{Eq:Morrey1}
        \l\|u\r\|_{\sC_{\Psi}}\lesssim \|u\|_{H^{\psi}_A}.
    \end{equation}
\end{theorem}
\begin{remark}
	From a perspective of theoretical completeness, exploring Sobolev-type inequalities for generalized Orlicz-Bessel potential spaces can provide a more comprehensive understanding of these function spaces. However, while this inequality is of interest, it is not directly utilized in the proof of our main results. Therefore, we have chosen to defer the investigation of this topic to a future paper.
\end{remark}
Our result implies the classical Morrey's inequality for Bessel potential space: 
\begin{corollary}\label{cor:morry-bp}
	\begin{enumerate}[(a)]
		\item Let $\alpha\in (0,1)$. Assume $\psi(R)=\phi(R^2)=R^\alpha$ and $A(t)=t^p$ with $p>d/\alpha$. Then  
		\[
		\|u\|_{C^{\alpha-\frac{d}{p}}}\lesssim \|u\|_{H^{\psi}_{A}}=\|u\|_{H^{\alpha}_{p}}. 
		\]
		\item 
		Let  $\psi(R)=\log(1+R^2)$ and $A(t)= \e^{t^{\beta}}-1$ with $\beta>1$. It holds that  
		\[
		\begin{aligned}
			&\|u\|_\infty + \sup_{|x-y|<\frac{1}{2}} |u(x)-u(y)| \cdot (-\log |x-y|)^{1-\frac{1}{\beta}}\\
			\lesssim& \|u\|_{H^{\psi}_A}= \|u+\log(I-\Delta) u\|_{A}.
		\end{aligned}
		\]
	\end{enumerate}
\end{corollary}
\begin{proof}
    (a) Let $\phi(r)=r^{\frac{\alpha}{2}}$. Since
    \[
        \rho_{\Psi}(r)=\frac{1}{\Psi(\frac{1}{r})}=\int_{\frac{1}{r}}^{\infty}  \frac{\d A^{-1}(t^d)}{\psi(t)}\asymp \int_{\frac{1}{r}}^\infty t^{\frac{d}{p}-\alpha-1}\,\d s\asymp r^{\alpha-\frac{d}{p}},
    \]
	the conclusion follows by Theorem \ref{Thm:Morrey1} and Theorem \ref{Thm:chart}. 
	
	(b) By definition, 
	\begin{equation*}
		A^{-1}(t)=\log^{\frac{1}{\beta}}(1+t). 
	\end{equation*} 
	For any $r\ll 1$, 
	\begin{align*}
		\rho_{\Psi}(r)=& \frac{1}{\Psi(\frac{1}{r})}=\int_{\frac{1}{r}}^\infty  \frac{\d A^{-1}(t^d)}{\psi(t)} \asymp \int_{\frac{1}{r}}^\infty \frac{t^{d-1}\log^{\frac{1}{\beta}-1}(1+t^d)}{(1+t^d)\log(1+t^2)}~\d t \\
		\asymp & \int_{r^{-d}}^\infty t^{-1} \log^{\frac{1}{\beta}-2} t\,\d t \asymp \l(\log \frac{1}{r}\r)^{\frac{1}{\beta}-1}.
	\end{align*}
	Again by Theorem \ref{Thm:Morrey1} and Theorem \ref{Thm:chart}, we get the desired result. 
\end{proof}

Now we give the proof for Theorem \ref{Thm:Morrey1} below. 
\begin{proof}[Proof of Theorem \ref{Thm:Morrey1}]
    Recall that $S$ is a subordinator with Laplace exponent $\phi$, and that $Z$ is the subordinate Brownian motion corresponding to $S$. For any \(\lambda>0\) and \(f\in C_0(\mR^d)\), set
    \[
        P_tf(x)=\bE f(x+Z_t)~ \mbox{ and } ~ G_\lambda f(x)=\int_0^\infty \e^{-\lambda t} P_tf(x)\,\d t. 
    \]
    Then 
    \begin{equation}\label{eq:resolvent1}
	 (\lambda+\phi(-\Delta)) G_\lambda f=f, \quad f\in C_0(\mR^d)
    \end{equation}
    and 
    \begin{equation}\label{eq:resolvent2}
    G_\lambda(\lambda+\phi(-\Delta))u=u  \mbox{ with } u=G_\lambda f, \quad f\in C_0(\mR^d). 
    \end{equation}
    Since \(f\in C_0 \cap L_A\) is dense in \(E_A\), the operator \(G_\lambda\) then can be extended to a mapping from \(E_A\) to \(H^\psi_A\). Moreover, \eqref{eq:resolvent1} still holds for all \(f\in E_A\), and \(H^\psi_A=G_\lambda E_A\). Thus, 
    \[
	G_\lambda (\lambda+\phi(-\Delta))u=u, \quad u\in H^\psi_A.  
    \]
    Let $f=u+\phi(-\Delta) u$. Then 
    \[
	u= G_1f= \int_0^\infty \e^{-t} P_t f\,\d t. 
    \]
    We need to prove 
    \begin{equation}\label{eq:Pi-u1}
        \|\Pi^\Psi_j u\|_\infty= \l\|\Pi^\Psi_j \int_0^\infty \e^{-t} P_t f\,\d t\r\|_\infty\leq C  2^{-j} \|f\|_{A}.
    \end{equation}
    Denote by $h_t$ and $p_t$ the transition probability density functions of Brownian motion $\sqrt{2}B_t$ and $Z_t$, respectively. Let $s_t(\d u)$ be the distribution of $S_t$.
	Let $\zeta_j(x)= (\Psi^{-1}(2^j))^d\varrho(\Psi^{-1}(2^j)x)$, where $\varrho$ is the smooth function in Section \ref{sec-IDD}. Thanks to H\"older's inequality \eqref{Eq:Holder}, we have 
	\begin{equation}\label{eq:Pi-u2}
		\begin{aligned}
			\|\Pi^\Psi_j u\|_\infty\leq& \l\| \Pi^\Psi_j \int_0^\infty \e^{-t} P_t f\,\d t  \r\|_\infty \leq \l\| \int_0^\infty \e^{-t}   p_t*(\zeta_{j+1}-\zeta_j)* f\,\d t  \r\|_\infty\\
			\lesssim& \|f\|_{A} \int_0^\infty \e^{-t} \| p_t*(\zeta_{j+1}-\zeta_j) \|_{A_{*}}\,\d t.  
		\end{aligned}
	\end{equation}
	Thus, to prove \eqref{eq:Pi-u1}, we need to estimate  $\| p_t*(\zeta_{j+1}-\zeta_j) \|_{A_{*}}$. 
	
	Let $R'>R\gg 1$ and $\varrho^R(x)= R^d\varrho(R x)$. Using the fact that 
	\[
	p_t(x)=\int_0^\infty h_u(x) s_t(\d u),
	\]
	we get 
	\begin{equation}\label{eq:morrey1}
		\begin{aligned}
			&p_t*(\varrho^{R'}  - \varrho^{R} ) (x)\\
			=& \int_{\mR^d}  p_t\l(x+z\r)R'^d\varrho(R' z)\,\d z -\int_{\mR^d}  p_t\l(x+z\r)R^d\varrho(R z)\,\d z \\
			=& \int_{\mR^d} \l[ p_t\l(x+\frac{z}{R'}\r)- p_t\l(x+\frac{z}{R}\r) \r] \varrho(z) ~ \d z \\
			=& \int_0^\infty\!\!\!\int_{\mR^d} \l| h_u\l(x+\frac{z}{R'}\r)-h_u\l(x+\frac{z}{R}\r) \r|   \varrho(z) ~s_t(\d u)\,\d z \\
			=& \int_0^\infty\!\!\!\int_{\mR^d} u^{-\frac{d}{2}} \l| h_1\l(\frac{x}{\sqrt{u}}+\frac{z}{\sqrt{u}R'}\r)-h_1\l(\frac{x}{\sqrt{u}}+\frac{z}{\sqrt{u}R}\r)\r| s_t(\d u)\varrho(z) ~ \d z. 
		\end{aligned}
	\end{equation}
	Here we use the fact that $h_u(x)=u^{-\frac{d}{2}} h_1(x/\sqrt{u})$. 
	By mean value theorem, one sees that 
	\begin{align*}
		&\l| h_1\l(\frac{x}{\sqrt{u}}+\frac{z}{\sqrt{u}R'}\r)-h_1\l(\frac{x}{\sqrt{u}}+\frac{z}{\sqrt{u}R}\r)\r|\\
		\leq & \l( \l| h_1\l(\frac{x}{\sqrt{u}}+\frac{z}{\sqrt{u}R'}\r)\r|+ \l| h_1\l(\frac{x}{\sqrt{u}}+\frac{z}{\sqrt{u}R}\r)\r| \r) \wedge\\ 
		&\l|\frac{z}{\sqrt{u}R'}-\frac{z}{\sqrt{u}R}\r| \int_0^1\l|\nabla h_1\l(\frac{x}{\sqrt{u}}+(1-\theta)\frac{z}{\sqrt{u}R'}+\theta \frac{z}{\sqrt{u}R}\r)\r| \,\d \theta.
	\end{align*}
	This implies  
	\begin{equation}\label{eq:morrey2}
		\begin{aligned}
			&u^{-\frac{d}{2}}\l\| h_1\l(\frac{\cdot}{\sqrt{u}}+\frac{z}{\sqrt{u}R'}\r)-h_1\l(\frac{\cdot}{\sqrt{u}}+\frac{z}{\sqrt{u}R}\r)\r\|_{A_{*}}\\
			\lesssim& \l\|u^{-\frac{d}{2}} h_1\l(\frac{\cdot}{\sqrt{u}}\r)\r\|_{A_{*}}\wedge \l|\frac{z}{\sqrt{u}R'}-\frac{z}{\sqrt{u}R}\r| \l\|u^{-\frac{d}{2}} \nabla h_1\l(\frac{\cdot}{\sqrt{u}}\r)\r\|_{A_{*}}\\
			\lesssim& \l( 1\wedge \frac{z}{\sqrt{u}R}\r)\,  \l(\l\|h_u\r\|_{A_{*}}+\l\|u^{-\frac{d}{2}}\nabla h_1\l(\frac{\cdot}{\sqrt{u}}\r)\r\|_{A_{*}}\r). 
		\end{aligned}
	\end{equation}
	Setting 
	\begin{equation}\label{eq:fA}
		f_{A_{*}}(u)=\| h_u \|_{A_{*}}+\l\|u^{-\frac{d}{2}}\nabla h_1\l(\frac{\cdot}{\sqrt{u}}\r)\r\|_{A_{*}}, 
	\end{equation} 
	and combining \eqref{eq:morrey1} and \eqref{eq:morrey2}, we get  
	\begin{equation}\label{eq:morrey3}
		\begin{aligned}
			&\|p_t*(\varrho^{R'}  - \varrho^{R} )\|_{A_{*}}\\
			\lesssim & \int_0^\infty\!\!\!\int_{\mR^d}  \l( 1\wedge \frac{z}{\sqrt{u}R}\r)\, ~\varrho(z) \, f_{A_{*}}(u) \,\d z\, s_t(\d u)\\
			\lesssim& \int_0^\infty \l(\frac{1}{\sqrt{u}R} \int_0^{\sqrt{u}R} |\varrho(r)| r^{d}\,\d r + \int_{\sqrt{u}R}^{\infty} |\varrho(r)| r^{d-1}\,\d r \r) \, f_{A_{*}}(u)\, s_t(\d u). 
		\end{aligned}
	\end{equation}
	Noting that $|\varrho(r)|\lesssim 1\wedge r^{-k} \, (\forall k>0)$, we have 
	\begin{equation}\label{eq:morrey4}
		\frac{1}{N} \int_0^{N} |\varrho(r)| r^{d}\,\d r \lesssim 1\wedge \frac{1}{N} \,\mbox{ and } \int_{N}^{\infty} |\varrho(r)| r^{d-1}\,\d r  \lesssim 1\wedge \frac{1}{N}, \  \forall N>0. 
	\end{equation}
	Plugging estimate \eqref{eq:morrey4} into  \eqref{eq:morrey3}, we get 
	\begin{equation}\label{eq:morrey5}
		\|P_t (\varrho^{R'}  - \varrho^{R} )\|_{A_{*}}\lesssim  \int_0^\infty\l(1\wedge \frac{1}{\sqrt{u}R}\r) \, f_{A_{*}}(u)\, s_t(\d u).
	\end{equation}
	So the main problem comes down to the estimation of $f_{A^*}(u)$. Let us estimate the first term on the right-hand side of \eqref{eq:fA}. By the scaling property of $h_u$, one sees that 
	\begin{equation}\label{eq:IA}
		\begin{aligned}
			I_{A_{*}}(h_u/\lambda)\leq 1 &\iff \int_{\mR^d} A_{*}\l(\frac{h_1\l({x}/{\sqrt{u}}\r)}{(\sqrt{u})^d\lambda}\r)\, \d x\leq 1\\ &\iff \int_{\mR^d} A_{*}\l(\frac{h_1\l( x \r)}{(\sqrt{u})^d\lambda}\r)\, \d x \leq u^{-\frac{d}{2}}. 
		\end{aligned}
	\end{equation}
	Let  $K=({4\pi u})^{-d/2}$. Basic calculation yields  
	\begin{equation*}
		\begin{aligned}
			&\int_{\mR^d} A_{*}\l(\frac{h_1\l( x \r)}{(\sqrt{u})^d \lambda}\r)\, \d x =  \int_{\mR^d} A_{*}\l(\frac{\e^{-\frac{|x|^2}{4}}}{(\sqrt{4\pi u})^d\lambda}\r)\, \d x \\
			=& c_d \int_0^\infty A_{*}\l(\frac{\e^{-r^2/4}}{(\sqrt{4\pi u})^d\lambda}\r) r^{d-1}\,\d r\\
			=& c_d \int_0^{\frac{K}{\lambda}} A_{*}(r) \l(\log \frac{K}{\lambda}-\log r\r)^{\frac{d}{2}-1} r^{-1}\,\d r\\
			=& c_d \int_0^{\frac{K}{\lambda}} \l(\log \frac{K}{\lambda}-\log r\r)^{\frac{d}{2}-1} r^{-1}\,\d r\int_0^r \d A_{*}(s) \\
			=& c_d \int_0^{\frac{K}{\lambda}} \d A_{*}(s) \int_s^{\frac{K}{\lambda}} \l(\log \frac{K}{\lambda}-\log r\r)^{\frac{d}{2}-1} r^{-1}\,\d r \\
			=& c_d \int_0^{\frac{K}{\lambda}} \l(\log \frac{K}{\lambda}-\log s\r)^{\frac{d}{2}} D_{-}A_{*}(s)~\d s, 
		\end{aligned}
	\end{equation*}
	where $D_{-}f$ is the left derivative of $f$. Noting that $D_{-}A_{*}$ is increasing and $D_{-}A_{*}(x)(y-x)\leq A_{*}(y)-A_{*}(x)$, 
	we get 
	\begin{align*}
		\int_{\mR^d} A_{*}\l(\frac{h_1\l( x \r)}{(\sqrt{u})^d \lambda}\r)\, \d x\leq& c_d\int_0^\infty u^{\frac{d}{2}}\e^{-u}~\d u\cdot  D_{-}A_{*}\l(\frac{K}{\lambda}\r)\frac{K}{\lambda}\\
		\leq& c_d \l[ A_{*}\l(\frac{2K}{\lambda}\r)-A_{*}\l(\frac{K}{\lambda}\r)\r]\leq c_d A_{*}\l(\frac{2K}{\lambda}\r).
	\end{align*}
	Using \eqref{eq:IA}, we have 
	\begin{align*}
		I_{A_{*}}(h_u/\lambda)\leq 1 \Longleftarrow A_{*}\l(\frac{2K}{\lambda}\r) \leq C K. 
	\end{align*}
	This yields 
	\begin{equation*}
		\|h_u\|_{A_{*}}=\inf\l\{\lambda: I_{A_{*}}(h_u/\lambda)\leq 1\r\}\leq \frac{2K}{A_{*}^{-1}(C K)} \lesssim \frac{u^{-d/2}}{A_{*}^{-1}(cu^{-d/2})},\quad c\geq 1. 
	\end{equation*}
	Similarly, we also have 
	\[
	\l\|u^{-\frac{d}{2}}\nabla h_1\l(\frac{\cdot}{\sqrt{u}}\r)\r\|_{A_{*}}\lesssim \frac{u^{-d/2}}{A_{*}^{-1}(c  u^{-d/2})}.
	\]
	Thus, there exists a constant $c>0$ such that 
	\begin{equation*}
		f_{A_{*}}(u)\lesssim \frac{u^{-d/2}}{A_{*}^{-1}(c u^{-d/2})}.
	\end{equation*}
	This together with \eqref{eq:morrey5} and \eqref{Eq:AA*} yields 
    \begin{align*}
        & \int_0^\infty \e^{-t} \|p_t* (\varrho^{R'}  - \varrho^{R} )\|_{A_{*}}\,\d t \\
        \lesssim& \int_0^\infty\e^{-t}\int_0^\infty \frac{1\wedge \frac{1}{\sqrt{u}R}}{u^{d/2}A_{*}^{-1}(cu^{-d/2})}~ s_t(\d u)\,\d t\\
        \lesssim&\int_0^{R^{-2}} A^{-1}(cu^{-d/2}) U_1(\d u)+ \frac{1}{R} \int_{R^{-2}}^\infty \frac{1}{\sqrt{u}} A^{-1}(cu^{-d/2}) U_1(\d u)\\
        =&:I_1+I_2. 
    \end{align*}
    where $U_1$ is the $1$- potential measure of $S_t$. 

    Recall that \(A(t)=\int_0^t \alpha(s) \d s\), we have  \(A^{-1}(t)=\int_0^t \beta(s) \d s \) with \(\beta(s)=\frac{1}{\alpha(A^{-1}(s))}\).  

    For \(I_1\), by Fubini theorem and \eqref{eq:phi-U}, we have 
	\begin{equation*}
		\begin{aligned}
			I_1 \lesssim& \int_0^{R^{-2}} \int_0^{cu^{-\frac{d}{2}}} \beta(s) ~ \d s~ U_0(\d u)\\
			=& \int_0^\infty\!\!\!\int_0^\infty \1_{\{u< ({s}/{c})^{-{2}/{d}}\wedge R^{-2}\}} \beta(s) ~ U_0(\d u) \d s \\
			=& \int_0^\infty \beta(s) U_0((0, ({s}/{c})^{-{2}/{d}}\wedge R^{-2})) \d s\\
			=& A^{-1}(cR^d) U_0((0,R^{-2})) + \int_{cR^d}^\infty U_0((0, ({s}/{c})^{-{2}/{d}})) \d A^{-1}(s)\\
			\overset{\eqref{eq:phi-U}}{\lesssim} & \frac{A^{-1}(cR^d)}{\psi(R)}+ \int_{cR^d}^\infty \frac{\d A^{-1}(s)}{\psi((s/c)^{1/d})}  \lesssim  \frac{A^{-1}(cR^d)}{\psi(R)}+ \int_R^\infty \frac{ \d A^{-1}(ct^d)}{\psi(t)}. 
		\end{aligned}
	\end{equation*}
    
    For \(I_2\), again by Fubini theorem, \eqref{eq:phi-U}, \eqref{eq:phi-int} and \eqref{eq:t-leq-R}, we get 
    \begin{equation*}
        \begin{aligned}
            I_2=& \frac{1}{R} \int_{R^{-2}}^\infty \frac{1}{\sqrt{u}} A^{-1}(cu^{-d/2}) U_1(\d u) \\
        =&\frac{1}{R} \int_0^\infty\!\!\!\int_0^\infty \1_{\{R^{-2}<u<(s/c)^{-2/d}\}} \frac{\beta(s)}{\sqrt{u}} ~  U_1(\d u) \d s\\
        \asymp& \frac{1}{R} \int_{(0,\infty)^3}  \1_{\{R^{-2}<u<(s/c)^{-2/d}\wedge t\}} \beta(s) t^{-\frac{3}{2}} U_1(\d u)  \d s \d t \\
        \lesssim & \frac{1}{R} \int_{(0,\infty)^2} \1_{\{t>R^{-2}\}}\1_{\{s<cR^d\}}  U_1(0, (s/c)^{-2/d}\wedge t) \beta(s) t^{-\frac{3}{2}} \d s \d t  \\
        \overset{\eqref{eq:phi-U}}{\lesssim} &  \frac{1}{R} \int_{(0,\infty)^2}\1_{\{t>R^{-2}\}}\1_{\{s<cR^d\}}  \frac{1}{1+\phi((s/c)^{2/d}\vee t^{-1})} t^{-\frac{3}{2}} \beta(s) ~ \d s \d t \\
        \lesssim & \frac{1}{R} \int_{R^{-2}}^\infty t^{-\frac{3}{2}} \l[\int_0^{ct^{-\frac{d}{2}}} \frac{\beta(s)}{1+\phi(t^{-1})} \d s+\int_{ct^{-\frac{d}{2}}}^{cR^d} \frac{\beta(s)}{1+\phi((s/c)^{2/d})} \d s\r] \d t \\
        \lesssim & \frac{1}{R} \int_0^R \frac{A^{-1}(cu^d)}{1+\psi(u)} \d u+ \frac{1}{R} \int_0^{cR^d} \frac{\beta(s)}{1+\psi((s/c)^{1/d})}  \int_{(\frac{s}{c})^{-\frac{2}{d}}}^\infty t^{-\frac{3}{2}} ~ \d t \d s \\
        \lesssim & A^{-1}(cR^d) \frac{1}{R}\int_0^R \frac{1}{1+\psi(u)} \d u + \frac{1}{R} \int_0^{cR^d} \frac{ s^{\frac{1}{d}} \d A^{-1}(s) }{1+\psi((s/c)^{1/d})} \\
        \overset{\eqref{eq:phi-int},\eqref{eq:t-leq-R}}{\lesssim} &  \frac{A^{-1}(cR^d)}{\psi(R)}. 
        \end{aligned}
    \end{equation*}
    Noting that 
    \[
        \int_R^\infty \frac{ \d A^{-1}(ct^d)}{\psi(t)} \gtrsim \int_R^{2R} \frac{ \d A^{-1}(ct^d)}{\psi(t)} \gtrsim \frac{A^{-1}(cR^d)}{\psi(2R)} \gtrsim \frac{A^{-1}(cR^d)}{\psi(R)}, 
    \]
    Therefore, there is a constant \(R_0>1\) such that for all $R'>R\geq R_0$, it holds that
    \begin{equation}\label{eq:Pt-A*}
	\begin{aligned}
            \int_0^\infty \e^{-t} \|P_t (\varrho^{R'}  - \varrho^{R} )\|_{A_{*}}\,\d t \lesssim& \int_R^\infty \frac{ \d A^{-1}(ct^d)}{\psi(t)}\lesssim \int_R^\infty \frac{ \d A^{-1}(ct^d)}{\psi(t)}\\
            \lesssim& \int_R^\infty \frac{ \d A^{-1}(t^d)}{\psi(t)}\overset{\eqref{eq:Psi}}{=}\frac{1}{\Psi(R)}. 
	\end{aligned}
    \end{equation}
	When \(j\gtrsim \log_2 R_0\), by letting $R'= \Psi^{-1}(2^{j+1})\gg 1$ and $R= \Psi^{-1}(2^j)\gg 1$, and plugging \eqref{eq:Pt-A*} to \eqref{eq:Pi-u2}, we obtain 
	\[
	\|\Pi_j^{\Psi} u\|_\infty \lesssim \frac{\|f\|_A}{\Psi\l(\Psi^{-1}(2^j)\r)}\lesssim 2^{-j} \|f\|_A. 
 	\]
    When \(j\lesssim \log_2 R_0\), in view of \eqref{eq:Pi-u2}, we have 
    \[
    \|\Pi_j^{\Psi} u\|_\infty \lesssim \|f\|_A \sup_{j\lesssim \log_2 R_0} \|\zeta_j\|_A \lesssim \|f\|_A \lesssim 2^{-j} \|f\|_A. 
    \]
    So we obtain \eqref{eq:Pi-u1}, and complete our proof. 
\end{proof}

\subsubsection{Morrey-type inequality for generalized Orlicz-Besov spaces} 
We also prove a Morrey inequality for generalized Orlicz-Besov spaces, which will be used in the proofs for our main results. 
\begin{theorem}\label{Thm:Morrey2}
	Assume $S$ satisfies \ref{Aspt-phi'-1}. If $A$ is an $N$-function satisfies $A(t)\geq  [\psi^{-1}(ct^{1+\eps})]^d$ for some $\eps>0$, then  
	\begin{equation}\label{Eq:Morrey2}
		\l\|u\r\|_{\sC_\psi^{\frac{\eps}{1+\eps}}}\lesssim \|u\|_{B^{\psi}_A}.  
	\end{equation}
\end{theorem}
\begin{proof}
        For any \(u\in B^{\psi}_A\) and \(j\geq -1\), put \(h_j=\varrho_{j+1}-\varrho_j \in L_{A_{*}}\) and  \(\widetilde{u}_j= \widetilde{\Pi}_j  u \in L_A\). Note that for any \(|x|\geq 2R \gg 1\), 
        \begin{align*}
        	|\Pi^{\psi}_j u (x)| =& \l| \int_{|y|\leq R } \widetilde{u}_j(y) h_j(x-y) \r|+\l|\int_{|y|> R } \widetilde{u}_j(y) h_j(x-y)  \r|\\
        	\lesssim & \|\widetilde{u}_{j}\|_{A} \|h_j\1_{B_R^c}\|_{A_{*}}+  \|\widetilde{u}_{j} \1_{B_R^c} \|_{A} \|h_j\|_{A_{*}} \to 0, \quad R\to \infty. 
        \end{align*}
        Therefore, \( \Pi^{\psi}_j u \in C_0\cap L_A\), for any \(u\in B^\psi_A\).  Similarly, one can verify that \((1+\phi(-\Delta)) \Pi^{\psi}_j u = (1+\phi(-\Delta))  \Pi^{\psi}_j u_j \in C_0\cap L_A\). Thus,  
        \[
        \Pi^{\psi}_j u=G_1(1+\phi(-\Delta)) \Pi^{\psi}_j u =  G_1 \Pi^{\psi}_j(1+\phi(-\Delta)) u, \quad u\in B^{\psi}_{A}. 
        \]
    Put \(f=(1+\phi(-\Delta)) u\). Following the proof for \eqref{eq:Pi-u2}, one can verify that 
    \begin{align*}
        \|\Pi^\psi_j u\|_\infty\leq& \l\|  \int_0^\infty \e^{-t} P_t~ \widetilde{\Pi}^\psi_j ~(\Pi^\psi_j f)\,\d t  \r\|_\infty \\
        \leq & \l\| \int_0^\infty \e^{-t}   p_t*(\widetilde{\varrho}_{j+1}-\bar{\varrho}_j)* (\Pi^\psi_j f)\,\d t  \r\|_\infty\\
        \lesssim& \|\Pi^\psi_j f\|_{A} \int_0^\infty \e^{-t} \| p_t*(\widetilde{\varrho}_{j+1}-\bar{\varrho}_j) \|_{A_{*}}\,\d t. 
    \end{align*}
    By our assumption $A(t)\geq [\psi^{-1}(ct^{1+\eps})]^d$, we have 
    \[
	A^{-1}(s)\leq C \psi^{\frac{1}{1+\eps}}(s^{\frac{1}{d}}). 
    \]
    Using this and following the proof for  \eqref{eq:Pt-A*}, we get 
    \begin{equation*}
	\begin{aligned}
            &\int_0^\infty \e^{-t} \| p_t*(\widetilde{\varrho}_{j+1}-\bar{\varrho}_j) \|_{A_{*}}\,\d t\\
            \lesssim &\int_{\frac{1}{2}\psi^{-1}(2^j)}^\infty   \frac{\d A^{-1}(s^d)}{\psi(s)} \lesssim  \int_{\frac{1}{2}\psi^{-1}(2^j)}^\infty  \frac{\psi^{\frac{1}{1+\eps}}(s) \psi'(s)}{\psi^2(s)}\,\d s\\
            \lesssim& \frac{1}{\psi^{\frac{\eps}{1+\eps}}(\frac{1}{2}\psi^{-1}(2^j))} \lesssim 2^{-\frac{\eps j}{1+\eps}}.
	\end{aligned}
    \end{equation*}
    Here we used the fact that for each $b>0$, 
	\[
	\lim_{R\to\infty} \frac{\psi\big(b \psi^{-1}(R)\big)}{R}=\lim_{R\to\infty} \frac{\psi\big(b \psi^{-1}(R)\big)}{\psi\big(\psi^{-1}(R)\big)} =b^\alpha. 
	\]
    Thus, 
    \[
	2^{\frac{\eps j}{1+\eps}} \|\Pi^\psi_j u\|_\infty \lesssim \|\Pi^\psi_j f\|_{A} \lesssim \|u\|_{B^{\psi,1}_A}. 
    \]
    So we complete the proof for \eqref{Eq:Morrey2}. 
\end{proof}

\section{Studies of Poisson equations}\label{Sec-Main}
In this section, we study the Poisson equation \eqref{Eq:PE} in the generalized Orlicz-Besov space. Our analysis is divided into two parts: In the first part, we explore the case where the coefficient $a$ is independent of $x$. Subsequently, we extended our findings to the more general case by using the method of frozen coefficients. 

\subsection{Spatially homogeneous case}
\begin{theorem}\label{Thm:Main0}
	Let $A$ be an $N$-function or $A=\infty$. 
	Assume \ref{Aspt:J-1} holds. Suppose that $a: \mR^d\to [0, \infty)$ is measurable, 
	\begin{equation}\label{cdt-k0}
		\int_{\mR^d} (1\wedge |z|^2)~ a(z)J(z)\,\d z<\infty, 
	\end{equation}
	and there are positive numbers $\rho_0, c_0 \in (0,1)$ such that 
	\begin{equation}\label{cdt-k1}
		a(z)\geq c_0, \ \mbox{ for all }z\in B_{\rho_0}. 
	\end{equation}
	Then there exists a constant $\lambda_0>0 $ such that for any $\lambda\geq \lambda_0$ and any $\alpha\in \mR$, 
	\begin{equation}\label{Eq:key0}
		\lambda \|u\|_{B^{\psi,\alpha}_{A}}+\|u\|_{B^{\psi,1+\alpha}_{A}}  \leq  C \|\lambda u-\cL u\|_{B^{\psi,\alpha}_{A}}, 
	\end{equation}
	provided that $u\in B^{\psi,1+\alpha}_{A}$. Here $C$ is a constant only depending on $d, \psi, \rho_0, c_0$ and $\alpha$. 
\end{theorem}
\begin{remark}
	We would like to emphasize that Theorem \ref{Thm:Main0} does not require the assumption that $\alpha\geq 0$.
\end{remark}

\begin{proof}
	Without loss of generality, we can assume that $c_0=1$ to simplify the analysis. By utilizing the condition in \eqref{cdt-k0}, we can identify a L\'evy process $Z$ whose infinitesimal generator corresponds to the operator $\cL$.  As presented in the proof for Theorem \ref{Thm:Morrey2}, for any \(u\in B^{\psi}_A\), we have 
	\[
	\Pi^{\psi}_j u=G_\lambda(\lambda - \cL) \Pi^{\psi}_j u =  G_\lambda \Pi^{\psi}_j(\lambda -\cL) u. 
	\]	 
	Using this and letting $f=\lambda u-\cL u$, we obtain 
	\begin{equation}\label{eq:low}
		\begin{aligned}
			\| \Pi_{j}^{\psi} u\|_{A}=&\l\| \Pi_{j}^{\psi} G_\lambda f \r\|_A = \l\| \Pi_{j}^{\psi}  \int_0^{\infty} \e^{-\lambda t} (P_t f)\d t \r\|_A \\
			\leq& \int_0^\infty \e^{-\lambda t} \|P_t 
			\Pi_{j}^{\psi} f\|_A\,\d t  \leq \lambda^{-1} \|\Pi_{j}^{\psi} f\|_A.   
		\end{aligned}
	\end{equation}
	
	This yields 
	\begin{equation}\label{eq:key0}
		\lambda \|u\|_{B^{\psi,\alpha}_{A}}\leq \|f\|_{B^{\psi,\alpha}_{A}}= \|\lambda u-\cL u\|_{B^{\psi,\alpha}_{A}}.
	\end{equation}
	Thus, it remains to show that the second term on the left-hand side of \eqref{Eq:key0} can be dominated by $C \|\lambda u-\cL u\|_{B^{\psi,\alpha}_{A}}$. 
	
	We divide the proof into several steps.
	
	\noindent{\bf Step 1.} 
	We first consider the case that $a\equiv 1$. In this case, $\cL$ is the infinitesimal generator of a subordinate Brownian motion $Z_t=\sqrt{2}B_{S_t}$. It is enough to show 
	\begin{equation}\label{eq:resolve_f}
		\| \Pi_{j}^{\psi} u\|_{A}=\l\| \Pi_{j}^{\psi}  \int_0^{\infty} \e^{-\lambda t} (P_t f)\d t \r\|_A \leq C 2^{-j} \|\Pi_{j}^{\psi} f\|_A, \quad j\geq 0.  
	\end{equation}
	Using the fact that $\widetilde{\Pi}^{\psi}_j\Pi_{j}^{\psi}=\Pi_{j}^{\psi}~(j\geq 0)$, we see that 
	\[
	\Pi_{j}^{\psi} (P_t f) = P_t (\Pi_{j}^{\psi} f)= P_t  (\widetilde \Pi_{j}^{\psi} \Pi_{j}^{\psi} f).
	\]
	Thus, 
	\begin{equation}\label{eq:ptf1}
		\|\Pi_{j}^{\psi} (P_t f)\|_A \leq  \| P_t (\widetilde \varrho_{j+1} -\bar \varrho_{j}) \|_1~ \|\Pi_{j}^{\psi} f\|_A, 
	\end{equation}
	due to Young's inequality \eqref{Eq:Young}. So our problem boils down to estimating $\| P_t (\widetilde \varrho_{j+1} -\bar \varrho_{j}) \|_1$ $(j\geq 0)$. For any $R'>R\gg 1$,  put 
	\[
	\varrho^R(x)= R^d \varrho(R x), \quad \varrho^{R'}(x)= (R')^d \varrho(R' x). 
	\]
	Employing the procedure deducing \eqref{eq:morrey3}, for  any $t>0$, we have 
	\begin{equation}\label{eq:ptf2}
		\begin{aligned}
			\| P_t (\varrho^{R'}  - \varrho^{R} ) \|_1\lesssim& \int_0^\infty\int_{\mR^d}  \l(1\wedge \frac{|z|}{\sqrt{u}R}\r) ~|\varrho(z)| \d z~s_t(\d u) \\
			\lesssim& \int_0^\infty \l(\frac{1}{\sqrt{u}R} \int_0^{\sqrt{u}R} |\varrho(r)| r^{d}\,\d r + \int_{\sqrt{u}R}^{\infty} |\varrho(r)| r^{d-1}\,\d r \r)  s_t(\d u) . 
		\end{aligned}
	\end{equation}
	Noting that $|\varrho(r)|\lesssim 1\wedge r^{-k} \, (\forall k>0)$, one sees that 
	\begin{equation}\label{eq:ptf3}
		\frac{1}{N} \int_0^{N} |\varrho(r)| r^{d}\,\d r \lesssim 1\wedge \frac{1}{N} \,\mbox{ and } \int_{N}^{\infty} |\varrho(r)| r^{d-1}\,\d r  \lesssim 1\wedge \frac{1}{N}, \  \forall N>0. 
	\end{equation}
	Combining \eqref{eq:ptf2} and \eqref{eq:ptf3}, we obtain  
	\begin{equation}\label{eq:ptf4}
		\begin{aligned}
			 &~\| P_t (\varrho^{R'}  - \varrho^{R} ) \|_1 \lesssim \int_0^\infty \l(1\wedge \frac{1}{\sqrt{u}R}\r) s_t(\d u) \\
			\leq &~  \bP \l(S_t\leq R^{-2}\r) + \frac{1}{R}\bE \l(S_t^{-\frac{1}{2}}\1_{\{S_t\geq R^{-2}\}}\r)\\
			=:&~ I_1(t,R)+I_2(t,R). 
		\end{aligned}
	\end{equation}
	For $I_1$, noting that $\e^{-R^2 S_t} \geq \e^{-1}\1_{\{S_t\leq R^{-2}\}}$, we have 
	\begin{equation*}
		I_1(t,R)=\bP \l(S_t\leq R^{-2}\r) \leq  \bE \exp( {1-R^2 S_t} ) \lesssim \e^{-t\psi(R)}.  
	\end{equation*}
	Thus, 
	\begin{equation}\label{eq:st1}
		\int_0^\infty \e^{-\lambda t}I_1(t,R)\,\d t \lesssim \frac{1}{\lambda +\psi(R)}. 
	\end{equation}
    For $I_2$, in virtue of Lemma \ref{Le-Inter-Karamata} (i) with $\sigma = -\frac{1}{2}$, $s = R^2$ and $f=1/(\lambda+\phi)$, and \eqref{eq:phi-U}, we have
    \begin{equation}\label{eq:st2}
	\begin{aligned}
            &\int_0^\infty \e^{-\lambda t} I_2(t, R)\,\d t=\frac{1}{R} \int_0^\infty \e^{-\lambda t}\bE \l(S_t^{-\frac{1}{2}}\1_{S_t\geq R^{-2}}\r)\,\d t\\
            =&\frac{1}{R} \int_{R^{-2}}^\infty u^{-\frac{1}{2}} \int_0^\infty \e^{-\lambda t}~\bP(S_t\in \d u)\,\d t =\frac{1}{R} \int_{R^{-2}}^\infty u^{-\frac{1}{2}}  U_\lambda (\d u)\\
            =& \frac{1}{R} \int_{(0,\infty)^2} \1_{\{R^{-2}<u<t\}} t^{-\frac{3}{2}} U_\lambda (\d u) \d t\\
            \lesssim& \frac{1}{R} \int_{R^{-2}}^\infty \frac{t^{-\frac{3}{2}}}{\lambda+\phi(t^{-1})} \d t \lesssim \frac{1}{\lambda+\psi(R)}. 
	\end{aligned}
    \end{equation}
    Here we use Karamata's  theorem in the last inequality. By \eqref{eq:ptf4}-\eqref{eq:st2}, we get 
	\begin{equation}\label{eq:ptf5}
		\int_0^\infty \e^{-\lambda t}\|P_t (\varrho^{R'}  - \varrho^{R} ) \|_1\,\d t\lesssim \frac{1}{\psi(R)}, \quad R'\geq  R\gg 1. 
	\end{equation}
  In view of \eqref{eq:ptf1} and \eqref{eq:ptf5},  there is a constant \(j_0\gg  1\) such that for any \(j\geq j_0\), it holds that 
	\begin{equation*}
		\begin{aligned}
				\| \Pi_{j}^{\psi} u\|_{A} = &\l\| \Pi_{j}^{\psi}  \int_0^{\infty} \e^{-\lambda t} (P_t f)\d t \r\|_A
			\overset{\eqref{eq:ptf1}}{\leq} \|\Pi_{j}^{\psi} f\|_A \int_0^\infty \e^{-\lambda t} \| P_t (\widetilde \varrho_{j+1} -\bar \varrho_{j}) \|_1\,\d t\\
			\overset{\eqref{eq:ptf5}}{\lesssim} &  \frac{\|\Pi_j^\psi f\|_A}{\psi\l(\frac{1}{2}\psi^{-1}(2^j)\r)}\lesssim {2^{-j}}\|\Pi_j^\psi f\|_A. 
		\end{aligned}
	\end{equation*}
	Here we choose $R'= 2\psi^{-1}(2^{j+1})$ and $R= \tfrac{1}{2}\psi^{-1}(2^j)$ in \eqref{eq:ptf5}, and use the fact that for each $b>0$, 
    \[
    \lim_{R\to\infty} \frac{\psi\big(b \psi^{-1}(R)\big)}{R}=\lim_{R\to\infty} \frac{\psi\big(b \psi^{-1}(R)\big)}{\psi\big(\psi^{-1}(R)\big)} =1. 
    \]
    For \(0\leq j < j_0\),  in view of \eqref{eq:low}, we have 
    \[
    \| \Pi_{j}^{\psi} u\|_{A}\lesssim \|\Pi_j^\psi f\|_A\lesssim 2^{-j_0} \|\Pi_j^\psi f\|_A \lesssim 2^{-j}  \|\Pi_j^\psi f\|_A . 
    \]
    So we complete the proof for \eqref{eq:resolve_f} when $a(x,z)=a(z)=1$.
	
	\noindent{\bf Step 2}. Now assume that $a$ satisfies \eqref{cdt-k1} with $c_0=1$ and $\rho_0=\infty$. 
	Let 
	\[
	\bar{\nu}(\d z)= J(z) \d z\  \mbox{ and } \ \widetilde{\nu}(\d z)=\underbrace{(a(z)-1)}_{\geq 0} J(z)\d z. 
	\]
	Since $\bar{\nu}$ and $\widetilde{\nu}$ are two L\'evy measures, there are two pure jump L\'evy processes $\bar Z$ and $\widetilde{Z}$ associated with $\bar \nu$ and $\widetilde \nu$, respectively. Let $\bar P_t$ and $\widetilde P_t$ be the semigroups corresponding to $\bar Z$ and $\widetilde Z$, respectively. 
	Noting that $\Pi_{j}^{\psi} P_t f = \Pi_{j}^{\psi} \bar P_t\widetilde P_t f$,
	by {\em Step 1}, one sees that 
	\begin{equation*}
		\begin{aligned}
			\l\| \Pi_{j}^{\psi}  \int_0^{\infty} \e^{-\lambda t} (P_t f)\d t \r\|_A \overset{\eqref{eq:ptf1}}{\leq} \|\Pi_{j}^{\psi} f\|_A \int_0^\infty \e^{-\lambda t} \| P_t (\widetilde \varrho_{j+1} -\bar \varrho_{j}) \|_1\,\d t \lesssim  2^{-j} {\|\Pi_j^\psi f\|_A}. 
		\end{aligned}
	\end{equation*}
	Thus, we get \eqref{eq:resolve_f} in the case that $a$ is bounded below by $c_0>0$. 
	
	\noindent{\bf Step 3.}  Now suppose $a$ only satisfies \eqref{cdt-k1} with $c_0=1$. Set 
	\[
	\bar{\cL} u:= \int_{B_{\rho_0}} (u(x+z)-u(x)) \l({ a(z)\1_{B_{\rho_0}}(z) +\1_{B_{\rho_0}^c}(z)} \r)J(z)\,\d z 
	\]
	and 
	\[
	R u:= \int_{\mR^d} (u(x+z)-u(x)) \l((a(z)-1)\1_{B_{\rho_0}^c}\r)J(z)\,\d z. 
	\]
	Then $\bar \cL$ satisfies \eqref{cdt-k1} with $c_0=1$ and $\rho_0=\infty$. By {\em Step 2} and noting that $\lambda u- \bar \cL u = \lambda u - \cL u + R u$, we get  
	\begin{equation}\label{eq:u-Ru}
		\|u\|_{B^{\psi,1+\alpha}_{A}}\lesssim \|\lambda u-\bar \cL u\|_{B^{\psi,\alpha}_{A}} \lesssim \|\lambda u- \cL u\|_{B^{\psi,\alpha}_{A}}+  \|R u\|_{B^{\psi,\alpha}_{A}}. 
	\end{equation}
	Since 
	\[
	\|R f\|_A \lesssim \|f\|_A \int_{|z|\geq \rho_0}  (1+|a(z)|)J(z)\,\d z \lesssim \|f\|_A, 
	\]
	one sees that 
	\begin{equation}\label{eq:Ru}
		\begin{aligned}
			\|Ru\|_{B^{\psi,\alpha}_{A}}=&\sup_{j\geq -1}2^{j\alpha}\|\Pi_{j}^{\psi} R u\|_{A}=\sup_{j\geq -1}2^{j\alpha} \|R \Pi_{j}^{\psi}  u\|_{A}\\
			\lesssim& \sup_{j\geq -1}2^{j\alpha} \|\Pi_{j}^{\psi} u\|_{A}=\|u\|_{B^{\psi,\alpha}_{A}}.
		\end{aligned}  
	\end{equation}
	Combining \eqref{eq:key0}, \eqref{eq:u-Ru} and \eqref{eq:Ru}, we obtain 
	\[
	\lambda\|u\|_{B^{\psi,\alpha}_{A}}+\|u\|_{B^{\psi,1+\alpha}_{A}} \leq  C \l( \|\lambda u-\cL u\|_{B^{\psi,\alpha}_{A}}+ \|u\|_{B^{\psi,\alpha}_{A}} \r).
	\]
	Choosing $\lambda_0=2C$, we obtain \eqref{Eq:key0} for all $\lambda\geq \lambda_0$. This completes our proof. 
\end{proof}

\subsection{Spatially inhomogeneous case}
Before proving our main results, we need to make some necessary preparations. Under assumption \ref{Aspt:J-1}, one can verify that there is an integer  $N \geq 1$ such that for all $s\geq 1$, 
\begin{equation}\label{eq:inversepsi}
	\psi^{-1}(2^{1-N } s)\leq \frac{3}{8}\psi^{-1}(s). 
\end{equation}
Define 
\[
S_k^{\psi} f= \sum_{l\prec k} \Pi_{l}^{\psi} f, \quad   T^{\psi}_{f}g= \sum_{k}S_{k}^{\psi} f \cdot \Pi_{k}^{\psi} g,\quad R^{\psi}(f,g)= \sum_{k\sim l} \Pi_{k}^{\psi} f \cdot \Pi_{l}^{\psi} g. 
\]
Here $l\prec k$ means $l< k-N$ and $k\sim l$ means $|k-l|\leq N$.  Thus, 
\[
f\cdot g= \sum_{k,l} \Pi_{k}^{\psi} f \cdot \Pi_{l}^{\psi} g =  T^{\psi}_{f}g +  T^{\psi}_{g}f+ R^{\psi}(f,g). 
\]

For $R_1, R_2\geq0$ with $R_1<R_2$,   denote \(D_{R_1, R_2}:=\{x\in\mR^d: R_1\leq |x| \leq R_2\}\). The following simple fact will be used frequently: 
For any two functions $\hat{f}$ and $\hat{g}$ whose supports are  
in $B_{R_0}$ and $D_{R_1,R_2}$,  respectively, then 
\begin{equation}\label{eq:support0}
{\rm supp} [\hat{f}*\hat{g}] \subseteq 
D_{(R_1-R_0)^+, R_2+R_0}. 
\end{equation}
Noting that \(\mathrm{supp}[\cF(S_k^\psi f)] \subseteq B_{\psi^{-1}(2^{k-N})} \) and \(\mathrm{supp}[\cF(\Pi_k^\psi g)] \subseteq D_{\frac{3}{4}\psi^{-1}(2^{k}), \psi^{-1}(2^{k+1})} \), 
in view of \eqref{eq:support0} and \eqref{eq:inversepsi}, we have 
\begin{align*}
    {\rm supp}\l[ \cF\l(S_k^\psi f\cdot \Pi_k^{\psi}g\r) \r]=&{\rm supp}\l[ \cF\l(S_k^\psi f\r)*\cF\l(\Pi_k^{\psi}g\r)\r]\subseteq D_{\frac{3}{8}\psi^{-1}(2^{k}),~ 2\psi^{-1}(2^{k+1})}\\
    \subseteq& D_{\psi^{-1}(2^{k-N}),~ \frac{3}{4}\psi^{-1}(2^{k+N})}. 
\end{align*}
This yields that  
\begin{equation}\label{eq:support1}
	\Pi_j^{\psi} (S_k^\psi f\cdot \Pi_k^{\psi}g)= 0,  \mbox{ if } |j-k|\geq N. 
\end{equation}
Similarly, one can also verify that 
\begin{equation}\label{eq:support2}
    \Pi_j^{\psi} \l(\sum_{k: k\sim l}\Pi_k^\psi f\cdot \Pi_l^{\psi}g\r)=0, \mbox{ if } j\geq k+2N. 
\end{equation}

\begin{lemma}
	Let $A$ be an $N$-function or $A=\infty$.  Let $a: \mR^d\times \mR^d\to \mR$ be a bounded measurable function (need not be positive). Assume the linear operator $\cL$ is given by \eqref{Eq:Oprt}, and \ref{Aspt:J-1} are satisfied. 
	\begin{enumerate}[(a)]
		\item 
		For any $\theta>0$, it holds that 
		\begin{equation}\label{Eq:Lu-A1}
			\|\cL u\|_{B^{\psi,0}_A} \leq C \|u\|_{B^{\psi,1}_{A}}\cdot\sup_{z\in \mR^d} \|a(\cdot,z)\|_{\sC^{\theta}_{\psi}}, 
		\end{equation} 
		where $C$ only depends on $d,c_0, \psi$ and $\theta$.
		\item 
		For any $\alpha>0$ and $\theta\in (0,\alpha)$, it holds that  \begin{equation}\label{Eq:Lu-A2}
			\|\cL u\|_{B^{\psi,\alpha}_A} \leq C \l(\|u\|_{B^{\psi,1+\alpha}_{A}}\cdot\sup_{z\in \mR^d} \|a(\cdot,z)\|_{\sC^{\theta}_{\psi}}+\|u\|_{B^{\psi,1+\theta}_{A}}\cdot\sup_{z\in \mR^d}\|a(\cdot,z)\|_{\sC^{\alpha}_{\psi}}\r),
		\end{equation}  
	\end{enumerate}
	where the constant $C$ only depends on $d,c_0, \psi,\alpha$ and $\theta$. 
\end{lemma}
\begin{proof}
	Set 
	\[
	\delta_z u(x)= u(x+z)-u(x), \quad a_{z}(x)=a(x,z). 
	\]
	By definition, 
	\begin{equation}\label{eq:PiLu}
		\begin{aligned}
			\Pi_{j}^{\psi} \cL u(x)=& \Pi_{j}^{\psi} \l( \int_{\mR^d} \delta_z  u(\cdot) \ a(\cdot,z)J(z)~ \d z\r) \\
			=& \int_{\mR^d}  \left (\Pi_{j}^{\psi} \sum_{k,l\geq -1} \Pi_{k}^{\psi} (\delta_z u) \cdot \Pi_{l}^{\psi} a_z \right)(x) \ J(z)\,\d z=:  \int_{\mR^d}  I_j(x,z)~J(z)\,\d z.
		\end{aligned}
	\end{equation}
	We drop the variable $x$ below for simplicity. By \eqref{eq:support1} and \eqref{eq:support2}, 
    \setlength{\arraycolsep}{-1pt}
	\begin{eqnarray*}
		|I_j(z)| &=&\left|\Pi_{j}^{\psi} \sum_{k, l\geq -1} \big[(\delta_z  \Pi_{k}^{\psi}  u)\cdot \Pi_{l}^{\psi} a_z \big] \right| \\
		&=& \left|\Pi_{j}^{\psi} \left(\sum_{k\prec l} \delta_z  \Pi_{k}^{\psi}  u\cdot \Pi_{l}^{\psi} a_z +\sum_{k\succ l} \delta_z  \Pi_{k}^{\psi}  u\cdot \Pi_{l}^{\psi} a_z +\sum_{k\sim l} \delta_z  \Pi_{k}^{\psi}  u\cdot \Pi_{l}^{\psi} a_z \right)\right|\\
		&\overset{\eqref{eq:support1}, \eqref{eq:support2}}{\leq}&  \sum_{k\prec l\sim j} \big|\delta_z  \Pi_{k}^{\psi}  u\cdot \Pi_{l}^{\psi} a_z \big|+ \sum_{l\prec k\sim j} \big| \delta_z  \Pi_{k}^{\psi}  u\cdot \Pi_{l}^{\psi} a_z \big| +\sum_{l\sim k; k> j-2N} \big|\delta_z  \Pi_{k}^{\psi}  u\cdot \Pi_{l}^{\psi} a_z \big|\\
		&=:& I^{(1)}_j(z)+I^{(2)}_j(z)+I^{(3)}_j(z).
	\end{eqnarray*}
	By mean value theorem, we get 
	\begin{align*}
		|\delta_z  \Pi_{k}^{\psi} u(x)|= \left| \int_0^1 z\cdot \nabla \Pi_{k}^{\psi} u(x+tz) 
		\d t\right| \leq  |z| \int_0^1 \l| \nabla \Pi_{k}^{\psi} u(x+tz) \r|\,\d t. 
	\end{align*}
	Thus, 
	\begin{equation*}
		|\delta_z  \Pi_{k}^{\psi} u(x)|\leq \l(|\Pi_{k}^{\psi} u(x)|+|\Pi_{k}^{\psi} u(x+z)|\r) \wedge \l(|z| \int_0^1 \l| \nabla \Pi_{k}^{\psi} u(x+tz) \r|\,\d t \r). 
	\end{equation*}
	This and \eqref{Eq:Bernstein} imply that 
	\begin{equation}\label{eq:I1}
		\begin{aligned}
			\|\delta_z  \Pi_{k}^{\psi}  u\cdot \Pi_{l}^{\psi} a_z \|_A \leq& \| \delta_z  \Pi_{k}^{\psi}  u \|_A ~ \sup_{z} \|\Pi_{l}^{\psi} a_z \|_\infty\\ 
			\leq& \sup_{z} \|\Pi_{l}^{\psi}a_z\|_{\infty}~ \|\Pi_{k}^{\psi} u\|_A \l(1\wedge |z| \psi^{-1}(2^{k+1}) \r).
		\end{aligned}
	\end{equation}
	(i) If $\alpha=0$. Combining inequalities \eqref{eq:I1}, \eqref{eq:int-zJ(z)} and \eqref{eq:int-J(z)}, we get 
	\begin{equation}\label{eq:Ij1}
		\begin{aligned}
			& \int_{\mR^d}  \|I_j^{(1)}(z)\|_A J(z)\,\d z\leq \sum_{k\prec l\sim j} \int_{\mR^d}\big\| \delta_z  \Pi_{k}^{\psi}  u\cdot \Pi_{l}^{\psi} a_z \big\|_{A} J(z)\,\d z\\
			\overset{\eqref{eq:I1}}{\leq} \ \ \, & 2^{-\theta j} \sup_{z}\|a_z\|_{\sC^\theta_{\psi}} ~ \|u\|_{B^{\psi,1}_{A}}~ \sum_{k\prec l\sim j}  2^{-k}  \\
			&\l( \psi^{-1}(2^{k+1}) \int_{|z|\leq \frac{1}{\psi^{-1}(2^{k+1})}} |z|J(z)\,\d z +\int_{|z|> \frac{1}{\psi^{-1}(2^{k+1})}} J(z)\,\d z \r)\\
			\overset{\eqref{eq:int-zJ(z)},\eqref{eq:int-J(z)}}{\lesssim} & \|u\|_{B^{\psi,1}_{A}}~\sup_{z}\|a_z\|_{\sC^\theta_{\psi}}~ 2^{-\theta j}\sum_{k\prec l\sim j} 
			2^{-k} \l(\frac{2^k}{\psi^{-1}(2^{k+1})}+2^k\r) \\
			\lesssim & \|u\|_{B^{\psi,1}_{A}}~\sup_{z}\|a_z\|_{\sC^\theta_{\psi}}  ~ j 2^{-\theta j}. 
		\end{aligned}
	\end{equation}
	Similarly, noting that $\theta>0$, we have 
	\begin{equation}\label{eq:Ij2}
		\begin{aligned}
			\int_{\mR^d}  {\|I_j^{(2)}(z)\|_A}J(z)\,\d z \leq& \sum_{l\prec k\sim j} \int_{\mR^d}\big\| \delta_z  \Pi_{k}^{\psi}  u\cdot \Pi_{l}^{\psi} a_z \big\|_{A} J(z)\,\d z\\
			\lesssim & \|u\|_{B^{\psi,1}_{A}}~\sup_{z}\|a_z\|_{\sC^{\theta}_{\psi}}  ~ \sum_{l\prec k\sim j} 2^{-\theta l} 
			2^{-k}  \l(\frac{2^k}{\psi^{-1}(2^{k+1})}+2^k\r) \\
			\lesssim& \|u\|_{B^{\psi,1}_{A}}~\sup_{z}\|a_z\|_{\sC^{\theta}_{\psi}}, 
		\end{aligned}
	\end{equation}
	and 
	\begin{equation}\label{eq:Ij3}
		\begin{aligned}
			\int_{\mR^d}  {\|I_j^{(3)}(z)\|_A}J(z)\,\d z
			\lesssim &  \|u\|_{B^{\psi,1}_{A}}\sup_{z}\|a_z\|_{\sC^{\theta}_{\psi}} \sum_{l\sim k; k> j-2N}  2^{-\theta l}
			2^{-k}    \l(\frac{2^k}{\psi^{-1}(2^{k+1})}+2^k\r) \\
			\lesssim& \|u\|_{B^{\psi,1}_{A}}~\sup_{z}\|a_z\|_{\sC^{\theta}_{\psi}}~2^{-\theta j}. 
		\end{aligned}
	\end{equation}
	Combining \eqref{eq:PiLu}, \eqref{eq:Ij1}, \eqref{eq:Ij2} and \eqref{eq:Ij3}, we obtain \eqref{Eq:Lu-A1}. 
	
	(ii) When $\alpha>0$, by \eqref{eq:I1} and similar discussions in (i), one can see that 
	\begin{equation*}
		\begin{aligned}
			\int_{\mR^d}  \|I_j^{(1)}(z)\|_A J(z)\,\d z
			\lesssim & 2^{-\alpha j} \sup_{z}\|a_z\|_{\sC^\alpha_{\psi}}  ~ \|u\|_{B^{\psi,1+\theta}_{A}} \sum_{k\prec l\sim j} 
			2^{-(1+\theta)k}   \l(\frac{2^k}{\psi^{-1}(2^{k+1})}+2^k\r) \\
			\lesssim & 2^{-\alpha j} \|u\|_{B^{\psi,1+\theta}_{A}}~\sup_{z}\|a_z\|_{\sC^\alpha_{\psi}}, 
		\end{aligned}
	\end{equation*}
	\begin{equation*}
		\begin{aligned}
			\int_{\mR^d}  {\|I_j^{(2)}(z)\|_A}J(z)\,\d z\lesssim & 2^{-\alpha j}\|u\|_{B^{\psi,1+\alpha}_{A}}~\sup_{z}\|a_z\|_{\sC^{\theta}_{\psi}}  ~ \sum_{l\prec k\sim j} 2^{-\theta l} 
			2^{-k}    \l(\frac{2^k}{\psi^{-1}(2^{k+1})}+2^k\r)  \\
			\lesssim& 2^{-\alpha j}\|u\|_{B^{\psi,1+\alpha}_{A}}~\sup_{z}\|a_z\|_{\sC^{\theta}_{\psi}},
		\end{aligned}
	\end{equation*}
	and 
	\begin{equation*}
		\begin{aligned}
			\int_{\mR^d}  {\|I_j^{(3)}(z)\|_A}J(z)\,\d z
			\lesssim & \|u\|_{B^{\psi,1+\alpha}_{A}} ~ \sum_{l\sim k; k> j-2N} 
			2^{-(1+\alpha)k}   \l(\frac{2^k}{\psi^{-1}(2^{k+1})}+2^k\r) ~\sup_{z}\|\Pi_{l}^{\psi}a_z\|_{\infty} \\
			\lesssim& 2^{-\alpha j}\|u\|_{B^{\psi,1+\alpha}_{A}}~\sup_{x,z}|a(x,z)|. 
		\end{aligned}
	\end{equation*} 
	Combining the above estimates, we get \eqref{Eq:Lu-A2}. 
\end{proof}

\begin{lemma}\label{Le-Main1}
	Let $A$ be an $N$-function or $A=\infty$.  Suppose that \ref{Aspt:J-1} and \ref{Aspt:a-1}-\ref{Aspt:a-2}  are satisfied, then there are universal constants $\eps_1>0$ and $\lambda_0>0$  such that for each $\lambda\geq 2\lambda_0$ and $f\in B_A^{\psi, \beta}$ with $\beta\in [0,\alpha]$, the Poisson equation \eqref{Eq:PE} has a unique solution $u\in B^{\psi,1+\beta}_{A}$, provided that $|a(x,z)-a(0,z)|\leq \eps_1$ and $\lambda\geq \lambda_0$. Moreover, 
	\begin{equation}\label{Eq:apriori}
		\lambda \|u\|_{B^{\psi,\beta}_{A}}+\|u\|_{B^{\psi,1+\beta}_{A}} \leq C \|f\|_{B^{\psi, \beta}_{A}}, \ \mbox{ for each } \beta\in [0,\alpha]. 
	\end{equation} 
	In particular, it holds that   
	\begin{equation}
		\lambda \|u\|_{C^{\alpha}_{\rho}}+\|u\|_{C^{1+\alpha}_{\rho}} \leq C \|f\|_{C_{\rho}^{\alpha}}, 
	\end{equation}
	if $|a(x,z)-a(0,z)|\leq \eps_1$ and $\lambda\geq \lambda_0$, and $f\in C_{\rho}^{\alpha}$. 
\end{lemma}
\begin{remark}
	Although the above analytic result relies on a special assumption regarding the oscillation of $a$, it is sufficient for us to establish our Theorem \ref{Thm:Main2}. In the case where $A(t)=\infty$ and $s>0$, the space $B^{\psi, s}_A$ can be identified as the generalized H\"older space $C^{s}_{\rho}$, which exhibits the following localization property:
	\[
	\|f\|_{C^{s}_{\rho}}\asymp_\eps \sup_{y\in \mR^d} \|f\eta((\cdot-y))\|_{C^{s}_{\rho}}, 
	\]
	where $\eta$ is some smooth cut-off function. The method of frozen coefficients, combined with the aforementioned result, can be employed to eliminate the requirement that the oscillation of $a$ is small. We will present a detailed proof of this approach for Theorem \ref{Thm:Main1}. However, it should be noted that the localization property can not extend to $B^{s}_{p,\infty}$ (see \cite{triebel1992theory}). As a result, it is unclear whether the assumption on the oscillation of $a$ can be entirely eliminated.
\end{remark}

\begin{proof}[Proof of Lemma \ref{Le-Main1}]
	When $a\equiv 1$, by the proof of Theorem \ref{Thm:Main0}, we see that if $f\in B^{\psi,\beta}_A$, then $G_\lambda f\in B^{\psi, 1+\beta}_A$ is a solution to \eqref{Eq:PE} ($\forall \lambda>0$), and estimate  \eqref{Eq:apriori} holds true. Through the continuity method, it is enough to prove \eqref{Eq:apriori}, under the premise that $u\in B^{\psi,1+\beta}_A$ is a solution to equation \eqref{Eq:PE}.   By interpolation theorem, we only need to show  \eqref{Eq:apriori} for $\beta=0$ and $\beta=\alpha$.
	
	Assume  $u\in B^{\psi,1+\alpha}_A$. Set $f=\lambda u-\cL u$,  
	\[
	a_0(z)=a(0,z) \ \mbox{ and }\ \cL_0 u(x) = \int_{\mR^d} (u(x+z)-u(x)) a_0(z)J(z)\,\d z. 
	\] 
	Thus,  
	\begin{equation*}
		\lambda u- \cL_0 u = (\cL-\cL_0) u +f. 
	\end{equation*}
	Using Theorem \ref{Thm:Main0}, we have 
	\begin{equation}\label{eq:u-f}
		\lambda \|u\|_{B^{\psi,\alpha}_{A}}+\|u\|_{B^{\psi,1+\alpha}_{A}}  \leq  C \l(\|(\cL-\cL_0) u\|_{B^{\psi,\alpha}_{A}}+\|f\|_{B^{\psi,\alpha}_{A}}\r). 
	\end{equation}
	Choosing $\theta\in (0, \alpha)$, we have 
	\[
	\|a(\cdot, z)-a_0(\cdot, z)\|_{\sC^{\theta}_{\psi}} \leq \|a(\cdot, z)-a_0(\cdot, z)\|_{\sC^{0}_{\psi}}^{1-\frac{\theta}{\alpha}}~\|a(\cdot, z)-a_0(\cdot, z)\|_{\sC^{\alpha}_{\psi}}^{\frac{\theta}{\alpha}} \leq c_0^{-1} \eps_1^{1-\frac{\theta}{\alpha}}.
	\]
	By \eqref{Eq:Lu-A2} and interpolation, we have 
	\begin{equation}\label{eq:L-L0}
		\begin{aligned}
			\|(\cL-\cL_0) u\|_{B^{\psi, \alpha}_A} \leq& C \eps_1^{1-\frac{\theta}{\alpha}} \|u\|_{B^{\psi, 1+\alpha}_A}+ C \|u\|_{B^{\psi,1+\theta}_A}\\
			\leq& C \eps_1^{1-\frac{\theta}{\alpha}} \|u\|_{B^{\psi, 1+\alpha}_A}+ C \|u\|_{B^{\psi,\alpha}_A}. 
		\end{aligned}
	\end{equation}
	Combining \eqref{eq:u-f} and \eqref{eq:L-L0}, and choosing $\lambda\gg 1$, we get 
	\[
	\lambda \|u\|_{B^{\psi,\alpha}_{A}}+\|u\|_{B^{\psi,1+\alpha}_{A}}  \leq C \eps_1^{1-\frac{\theta}{\alpha}} \|u\|_{B^{\psi, 1+\alpha}_A} +  C \|f\|_{B^{\psi, \alpha}_A}.
	\]
	Choosing $\eps_1$ sufficiently small, so that $C\eps_1^{1-\frac{\theta}{\alpha}}\leq 1/2$, then we obtain \eqref{Eq:apriori} for $\beta=\alpha$. 
	
	The case $\beta=0$ can be proved by using  \eqref{Eq:Lu-A1} and following the same procedure above.  So we complete our proof. 
\end{proof}

Now we are at the point of proving Theorem \ref{Thm:Main1}.
\begin{proof}[Proof of Theorem \ref{Thm:Main1}]
	Let $\chi$ be the nonnegative smooth function with compact support defined in Section \ref{sec-IDD}. For fixed $x_0\in \mR^d$, define
	\[
	\chi_\eps^{x_0}(x):=\chi\Big(\frac{x-x_0}{\eps}\Big) \mbox{ and }  a^{x_0}_\eps(x,z):=[a(x,z)-a(x_0,z)]\chi_\eps (x). 
	\]
	By definition, $a_\eps^{x_0}$ satisfies that 
	\begin{equation}\label{eq:Keps1}
		|a_\eps^{x_0}(x, z)|\lesssim 1
	\end{equation}
	and for every  $z\in \mR^d$ and $|x-x'|< 1$,  
	\begin{equation}\label{eq:Keps2}
		|a_\eps^{x_0}(x, z)-a_\eps^{x_0}(x',z)|\leq C_{\eps} \rho^{\alpha}(|x-x'|). 
	\end{equation} 
	For simplicity, we omit the superscript $x_0$ below. Let $v=u\chi_\eps $, 
	\[
	\delta_z f(x)= f(x+z)-f(x) \ \mbox{ and }\  \cL_0 u(x)= \int_{\mR^d} \delta_z  f(x)~a(x_0,z) J(z)\,\d z.
	\]
	We have
	\begin{equation}\label{loc}
		\begin{split}
			&\lambda v- \cL_0 v \\
			=&[f\chi_\eps -u\cL_0  \chi_\eps ]+(\cL u-\cL_0 u)\chi_\eps  -\l[\cL_0(u\chi_\eps )-(\cL_0 u)\chi_\eps -u(\cL_0 \chi_\eps )\r]\\
			=&: \sum_{i=1}^3 I_\eps^{(i)}. 
		\end{split}
	\end{equation}
	Obviously,
	\begin{equation}\label{eq:loc1}
		\| I_\eps^{(1)}\|_{C_{\rho}^{\alpha}} {\leq}\|f\chi_\eps \|_{C_{\rho}^{\alpha}}+\|u\cL_0\chi_\eps \|_{C_{\rho}^{\alpha}} \lesssim_\eps \|f\|_{C_{\rho}^{\alpha}}+\|u\|_{C_{\rho}^{\alpha}} .
	\end{equation}
	Noting that 
	\[
	I_\eps^{(2)}(x)= (\cL_0 u(x)-\cL u(x))\chi_\eps (x)=-\int_{\mR^d} \delta_z  u(x) a_\eps(x,z)\,J(z)\,\d z, 
	\]
	then by \eqref{eq:Keps1}, \eqref{eq:Keps2} and \eqref{Eq:Lu-A2}, for any $\theta\in(0,\alpha)$
	\begin{equation}\label{eq:loc2}
		\|I_\eps^{(2)}\|_{C_{\rho}^{\alpha}} 
		\leq  c_1\rho^\alpha(\eps)  \|u\|_{C^{1+\alpha}_{\rho}}+  C \|u\|_{C^{1+\theta}_{\rho}}, 
	\end{equation}
	where the constant $c_1$ is independent with $\eps$. For $I_\eps^{(3)}$, by definition 
	\begin{equation}\label{w-eps2}
		I_\eps^{(3)}(x)
		=\int_{\mR^d}\delta_z \chi_\eps (x)\ \delta_z u (x)~a(x_0,z) J(z)\,\d z
	\end{equation}
	and
	\begin{equation}\label{delta-w}
		\begin{split}
			I_\eps^{(3)}(x)-I_\eps^{(3)}(y)=& \int_{\mR^d} \delta_z \chi_\eps  (x)\big[\delta_z u (x)-\delta_z u (y)\big]~a(x_0,z)J(z)\,\d z\\
			&+ \int_{\mR^d} \big[\delta_z \chi_\eps  (x)-\delta_z \chi_\eps (y)\big]\delta_z u(y)~a(x_0,z)J(z)\,\d z.
		\end{split}
	\end{equation}
	By \eqref{w-eps2}, we have 
	\begin{align*}
		\big|I_\eps^{(3)}(x)|&\leq\|u\|_{\infty} \l( \int_{|z|\leq 1} \|\nabla \chi_\eps \|_{\infty} |z|a(x_0,z) J(z)\,\d z+2\int_{|z|>1} \|\chi_\eps \|_{\infty}  a(x_0,z) J(z)\,\d z \r)\\
		&\lesssim_\eps  \|u\|_{\infty}.
	\end{align*}
	By \eqref{delta-w}, one can see that 
	\begin{equation*}
		\begin{split}
			&\big|I_\eps^{(3)}(x)-I_\eps^{(3)}(y)\big|\\
			\lesssim_\eps& \rho^{\alpha}(|x-y|)~\|u\|_{C_{\rho}^{\alpha}} \l(\|\nabla \chi_\eps \|_{\infty} \int_{|z|\leq 1} |z| J(z)\,\d z + \|\chi_\eps \|_{\infty} \int_{|z|>1} J(z)\,\d z \r)\\
			&+ |x-y|\l(\|\nabla^2\chi_\eps \|_{\infty} \|u\|_{\infty} \int_{|z|\leq 1} |z|J(z)\,\d z+  \|\nabla\chi_\eps \|_{\infty}  \|u\|_{\infty} \int_{|z|>1} J(z)\,\d z \r)\\
			\lesssim_\eps& \rho^{\alpha}(|x-y|)~\|u\|_{C_{\rho}^{\alpha}}. 
		\end{split}
	\end{equation*}
	Therefore, 
	\begin{equation}\label{eq:loc3} \|I_\eps^{(3)}\|_{C_{\rho}^{\alpha}}\lesssim_\eps\|u\|_{C_{\rho}^{\alpha}}. 
	\end{equation}
	Using Theorem \ref{Thm:Main0}, and combining \eqref{eq:loc1}, we obtain that  
	\eqref{eq:loc2} and \eqref{eq:loc3}, 
	\begin{equation*}
		\begin{aligned}
			&\lambda \|u\|_{C^{\alpha}_{\rho}(B_{\eps/2}(x_0))}+\|u\|_{C^{1+\alpha}_{\rho}(B_{\eps/2}(x_0))}\\
			\leq& C \|v\|_{C^{1+\alpha}_{\rho}}\leq C \|v\|_{C_{\rho}^{\alpha}}+ C \|\cL_0 v\|_{C_{\rho}^{\alpha}}\\
			\leq& c_2\rho^\alpha(\eps)\|u\|_{C^{1+\alpha}_{\rho}}+ C \|u\|_{C^{1+\theta}_{\rho}}+C \|u\|_{C_{\rho}^{\alpha}}+C\|f\|_{{C^{\alpha}_{\rho}}},
		\end{aligned}
	\end{equation*}
	for all $\lambda\geq \lambda'_0(d, \alpha, \psi, c_0,\rho_0)>0$. Using the fact that 
	\[
	\|u\|_{C^{1+\alpha}_{\rho}} \leq \sup_{x_0\in \mR^d}\|u\|_{C^{1+\alpha}_{\rho}(B_{\eps/2}(x_0))}+ C_\eps \|u\|_\infty, 
	\]
	we obtain 
	\[
	\|u\|_{C^{1+\alpha}_{\rho}(B_{\eps/2}(x_0))}
	\leq c_2 \rho^\alpha(\eps) \sup_{x_0\in \mR^d}\|u\|_{C^{1+\alpha}_{\rho}(B_{\eps/2}(x_0))}+ C\l(\|u\|_{C^{1+\theta}_{\rho}}+ \|u\|_{C_{\rho}^{\alpha}}+\|f\|_{{C^{\alpha}_{\rho}}}\r).
	\]
	We then fix $\eps_0$ sufficiently small, such that
	$c_2\rho^\alpha(\eps)\leq 1/2$, so we arrive 
	\begin{align*}
		\sup_{x_0\in\mR^d} \l(\lambda \|u\|_{C^{\alpha}_{\rho}(B_{\eps_0/2}(x_0))}+ \|u\|_{C^{1+\alpha}_{\rho} (B_{\eps_0/2}(x_0))}\r) \leq C\l(\|u\|_{C^{1+\theta}_{\rho}}+ \|u\|_{C_{\rho}^{\alpha}}+\|f\|_{C^{\alpha}_{\rho}}\r).
	\end{align*}
	This yields 
	\begin{align*}
		\lambda\|u\|_{C_{\rho}^{\alpha}}+\|u\|_{C^{1+\alpha}_{\rho}} \leq& C_{\eps_0} \sup_{x_0\in\mR^d} \l(\lambda \|u\|_{C^{\alpha}_{\rho}(B_{\eps_0/2}(x_0))}+\|u\|_{C^{1+\alpha}_{\rho}(B_{\eps_0/2}(x_0))}\r)\\
		\lesssim&  \l(\|u\|_{C^{1+\theta}_{\rho}}+ \|u\|_{C_{\rho}^{\alpha}}+\|f\|_{{C^{\alpha}_{\rho}}}\r). 
	\end{align*}
	By interpolation, one sees that 
	\[
	\lambda\|u\|_{C_{\rho}^{\alpha}}+\|u\|_{C^{1+\alpha}_{\rho}}\lesssim \|f\|_{{C^{\alpha}_{\rho}}}, 
	\]
	for all $\lambda\geq  \lambda_0(d, \alpha, \psi, c_0,\rho_0)>0$. 
\end{proof}

\begin{proof}[Proof of Corollary \ref{Cor:Log}]
	Let $S$ be the gamma subordinator, whose Laplace exponent is given by  $\phi(s)=\log(1+s)$. By Proposition \ref{Prop:j-regular}, one sees $j(r)\asymp r^{-d}, r\to0$. Let 
	\[
	a'(x,z)= |z|^{-d}j^{-1}(|z|)\1_{B_1}(z)~a(x,z). 
	\]
	Then  \(a'(x,z)\geq c_0\), and  $\|a'(\cdot,z)\|_{C_{\rho}^{\alpha}}\lesssim \|a(\cdot,z)\|_{C_{\rho}^{\alpha}}\leq C$, where  $\rho(r)=\frac{1}{\psi(r^{-1})}=\frac{1}{\log(1+r^{-2})}$. Noting that  
	\[
	L u(x)= \int_{B_1} \l(u(x+z)-u(x)\r)\frac{a(x,z)}{|z|^d}\,\d z=\int_{\mR^d} \l(u(x+z)-u(x)\r)~ a'(x,z)J(z)\,\d z, 
	\]
	by Theorem \ref{Thm:Main1}, we obtain our assertion. 
\end{proof}

\section{Martingale problem}\label{Sec-MP}
In this section, we study the martingale problem associated with $\cL$. We need to introduce some necessary terminology.

Let $D= D([0, \infty), \mR^d)$ be the Skorokhod space of c\`adl\`ag $\mR^d$-valued trajectories 
and let $X_t=X_t(\omega)=\omega_t~( \omega\in D)$, be the canonical process on it. Set 
\[
\cD_t=\bigcap_{\eps>0}\sigma(X_s, 0\leq s\leq t+\eps), \quad \cD=\sigma(\cD_t, t\geq 0). 
\]

\begin{definition}[Martingale problem]
	Let $\mu$ be a probability measure on $\mR^d$. A probability measure $\mP$ on $(D, \cD)$ is said to be a solution to the martingale problem for $(\cL, \mu)$ if $\mP\circ X_0^{-1}=\mu$ and for each $f\in C_b^\infty$, 
	\[
	M^f_t:= f(X_t)-f(x)-\int_0^t \cL f(X_s)\,\d s \ \mbox{ is a $\cD_t$-martingale under $\mP$.} 
	\]
\end{definition}

Below we briefly introduce  the SDE   corresponding  to the martingale problem $(\cL, \mu)$. Let $(\Omega, \bP, \cF)$ be a probability space and $N(\d r, \d z, \d s)$ be a Poisson random measure on $\mR_+\times \mR^d\times \mR_+$ with intensity measure $\d r\,J(z)\d z\,\d s$.
Consider the following SDE driven by Poisson random measure $N$:
\begin{equation}\label{Eq:Nsde}
	X_t=x+\int_0^t \int_{\mR^d}\int_0^\infty z\1_{ [0,a(X_{s-},z))}(r) N(\d r, \d z, \d s).
\end{equation}
If $|a(x, z)-a(x,z)|\leq C |x-y|$ for any $x, y\in \mR^d$, then 
\begin{equation*}
	\begin{aligned}
		\int_{\mR^d} \int_0^\infty |z|^2 \l( \1_{[0, a(x,z))}(r)-\1_{[0, a(y,z))}(r) \r) ^2 \d r~J(z)\d z \leq C |x-y|^2. 
	\end{aligned}
\end{equation*} 
This implies \eqref{Eq:Nsde} admits a unique strong solution (see \cite{applebaum2009levy}). By It\^o's formula, we see that 
\begin{align*}
	f(X_t)-f(x)=&\int_0^t \int_{\mR^d}\int_0^\infty\l( f\l(X_{s-}+  z\1_{[0,a(X_{s-},z))}(r)\r) - f(X_{s-}) \r) N(\d r, \d z, \d s)\\
	=&\int_0^t  \int_{\mR^d}\l(f(X_{s-}+z) -f(X_{s-})\r){a(X_{s-},z)} ~ J(z)\d z~\d s \\
	&+ \int_0^t  \int_{\mR^d}  \int_0^\infty  \l( f(X_{s-}+z) -f(X_{s-})\r)\1_{[0, a(X_{s-},z)]}(r) \widetilde{N}(\d r,\d z,\d s),\\
	=& \int_0^t \cL f(X_{s-})\d s + M^f_t, 
\end{align*}
where $\widetilde{N}(\d r, \d z, \d t)=N(\d r, \d z, \d t) -\d r J(z)\d z \d t $ and $M^f_t$ is a martingale. Therefore, when $x\mapsto a(x,z)$ is Lipschitz continuous (uniformly in $z$), $\bP\circ X^{-1}$ is a martingale solution to $(\cL, \delta_x)$. 

\smallskip

Now we are at the point of proving Theorem \ref{Thm:Main2}. 
\begin{proof}[Proof of Theorem \ref{Thm:Main2}]
	\noindent {\bf Existence:} The proof for the existence is standard, for the convenience of the reader, we
	give the details here. Let $\eta\in C^\infty_c(B_1)$ satisfying $\int \eta=1$. Set $\eta^n(x)=n^d\eta(nx)$ and $a^n(x,z)=\int a(x-y,z)\eta_n(y)\,\d y$. Let $X^n_t$ be the unique strong solution to  \eqref{Eq:Nsde} ($a$ is replaced by $a^n$), and $\mP^n_x := \bP\circ( X^n)^{-1}$. By the discussion above, $\mP^n_x$ is a martingale solution associated with $(\cL^n, \delta_x)$, and 
	\[
	\cL^n f(x)=\int_{\mR^d} \l(f(x+z)-f(x)\r) a^n(x,z)J(z)\,\d z. 
	\]
	We claim that $\{\mP^n_x\}$ is tight in $\sP(D)$. Then upon taking a subsequence, still denoted by $n$, we can assume that $\mP_x^n\Rightarrow \mP_x$. For any $0\leq s_1\leq \cdots \leq s_k \leq s\leq t$, $f\in C_c^\infty(\mR^d)$ and $h_i\in C_c^\infty(\mR^d)$, $i\in \{1,2\cdots,k\}$, we have  
	\begin{align*}
		&\bE_{\mP_x} \l\{\l[f(X_t)-f(X_s) -\int_s^t \cL f(X_u)\,\d u\r] \prod_{i=1}^k h_i(X_{t_i}) \r\}\\
		=& \lim_{n\to\infty} \bE \l\{\l[f(X^n_t)-f(X^n_s) -\int_s^t \cL f(X_u^n)\,\d u\r] \prod_{i=1}^k h_i(X_{t_i}^n) \r\}\\
		=& \lim_{n\to\infty} \bE \l\{\l[f(X^n_t)-f(X^n_s) -\int_s^t \cL^nf(X_u^n)\,\d u\r] \prod_{i=1}^k h_i(X_{t_i}^n) \r\}\\
		&+ \lim_{n\to\infty} \bE \l\{\l[\int_s^t (\cL^n-\cL) f(X_u^n)\,\d u \r]\prod_{i=1}^k h_i(X_{t_i}^n) \r\}=0. 
	\end{align*}
	Here we use the fact that $\bP\circ (X^n)^{-1}$ is a solution to $(\cL^n, \delta_x)$, and $(\cL^n-\cL)f\to 0$ uniformly (since $a^n\to a$ uniformly). Thus, $\mP_x$ is a martingale solution to $(\cL, \delta_x)$. 
	
	It remains to show $\{\mP^n_x\}$ is tight. Given a bounded stopping time $\tau$, we define 
	\[
	N_{\tau}(\d z) = \int_\tau^{\tau+1} \int_0^{c_0^{-1}} N(\d r, \d z, \d s). 
	\]
	$N_{\tau}(\d z)$ is a Poisson random measure on $\mR^d$ with intensity measure $c_0^{-1}J(z)\d z$. 
	\begin{align*}
		&\bP (|X^n_{\tau+\delta}-X^n_{\tau}|>\eps) \\
		\leq& \bP (|X^n_{\tau+\delta}-X^n_{\tau}|>\eps; N_{\tau}(B_M^c)=0)+ \bP(N_{\tau}(B_M^c)\geq 1)\\
		\leq & \eps^{-1}\bE \l|\int_{\tau}^{\tau+\delta}\int_{B_M}\int_0^{a(X_{s-},z)} z N(\d r,\d z, \d s)\r|+ \bP(N_{\tau}(B_M^c)\geq 1)\\
		\leq& \frac{\delta}{c_0\eps} \int_{|z|\leq M} |z|J(z)\d z+ c_0^{-1}\int_{|z|>M} J(z)\d z. 
	\end{align*}
	Letting $\delta\to0$ and then $M\to0$, we get $\sup_{n, \tau}\bP (|X^n_{\tau+\delta}-X^n_{\tau}|>\eps)\to 0 ~(\delta\to 0)$. Similarly, one can also see that $\sup_{n}\bP(|X^n|>M)\to 0 ~(M\to \infty)$. By Aldous tightness criterion, we get the tightness of $\{\mP^n_x\}$. 
	
	\noindent  {\bf Uniqueness:} For any $f\in C_b^\infty$, consider equation \eqref{Eq:PE} in $C_\rho^{1+\alpha}$. By Theorem \ref{Thm:Main1}, equation \eqref{Eq:PE} admits a unique solution $u\in C^{1+\alpha}_{\rho}$. Suppose that $(\mP^{i}_{x}, X_t)\ (i=1,2)$ are two solutions to martingale problem $(\cL, \delta_x)$, one can easily  verify that 
	\[
	M^u_t= u(X_t)-u(x)-\int_0^t \cL u(X_t)\,\d t
	\]
	is a martingale under $\mP_{x}^{i}$. This implies 
	\[
	\begin{aligned}
		\d ( \e^{-\lambda t} u(X_t) )=& - \e^{-\lambda t} (\lambda-\cL) u(X_t)\,\d t + \e^{-\lambda t}\,\d M^u_t\\
		=& -\e^{-\lambda t}f(X_t)\,\d t+ \e^{-\lambda t}\d M^u_t. 
	\end{aligned}
	\]
	Taking expectation, one sees 
	\[
	u(x)= \mE_{x}^{i} \int_0^\infty \e^{-\lambda t} f(X_t)\,\d t, \quad i=1,2. 
	\]
	Thanks to \cite[Theorem 4.2]{ethier2009markov}, $\mP_{x}^{1}=\mP_{x}^{2}$.  
	
	\noindent {\bf Krylov-type estimate:} Let $\delta\in (0,1\wedge \rho_0/10)$ such that 
	\[
	\rho^\alpha(2\delta)<\eps_1,  
	\]
	where $\eps_1$ is the same constant in Lemma \ref{Le-Main1}. For each $y\in \mR^d$, let
	\[
	a^y(x,z)= \l\{
	\begin{aligned}
		&a(x,z) \quad &\mbox{ if } x\in B_\delta(y), \\
		&a\l(y+\frac{\delta^2(x-y)}{|x-y|^2},z\r) \quad &\mbox{ if } x\in B_\delta^c(y). 
	\end{aligned}
	\r.
	\]
	Then $a^y$ satisfies \eqref{aspt:a-lower},  \eqref{Aspt:a-holder} and 
	\[
	|a^y(x,z)-a^y(y,z)|<\eps_1.
	\]
	Let 
	\[
	\cL^y u(x)=\int_{\mR^d} \l(u(x+z)-u(x)\r)~a^y(x,z)J(z)\,\d z. 
	\]
	Since $a^y$ also satisfies conditions \ref{Aspt:a-1} and \ref{Aspt:a-2}, by Theorem \ref{Thm:Main2} (a), for each $y\in \mR^d$, the martingale problem  $(\cL^y, \delta_y)$ admits a unique solution $(\mQ_y, X_t)$. Moreover, for each $y\in \mR^d$ and $f\in C_c^\infty$, by the proof for Theorem \ref{Thm:Main2} (a), we have 
	\begin{equation}\label{eq:EQ1}
		\bE_{\mQ_y}\int_0^\infty \e^{-\lambda t} f(X_t)\,\d t = u_{a^y}(y), 
	\end{equation}
	where $u_{a^y}$ is the solution to \eqref{Eq:PE} with $\cL$ replaced by $\cL^y$. Noting that $a^y$ meets all the assumptions in Lemma \ref{Le-Main1}, we have 
	\begin{equation}\label{eq:u-kappay}
		\|u_{a^y}\|_{B^{\psi,1}_A}\leq C \|f\|_{B^{\psi,0}_A}\leq C \|f\|_A. 
	\end{equation}
	Combining and \eqref{eq:EQ1} and \eqref{eq:u-kappay} and using Theorem \ref{Thm:Morrey2}, we get
	\begin{equation}
		\bE_{\mQ_y}\int_0^\infty \e^{-\lambda t} f(X_t)\,\d t \overset{\eqref{eq:EQ1}}{\leq} C \|u_{a^y}\|_{\sC_\psi^{\frac{\eps}{1+\eps}}} \overset{\eqref{Eq:Morrey2}}{\leq} C \|u_{a^y}\|_{B^{\psi,1}_A} \overset{\eqref{eq:u-kappay}}{\leq} C  \|f\|_A. 
	\end{equation}
	
	Now let $\tau_1=\inf\{t>0: |X_t-X_0|>\delta\}$ and $\tau_{k+1}=\tau_{k}+\tau_1\circ \theta_{\tau_{k}}$. It holds that 
	\[
	\mQ_y|_{\sF_{\tau_1}}= \mP_y|_{\sF_{\tau_1}}, 
	\]
	due to the fact that  $a^y(\cdot, z)|_{B_{\delta}(y)}=a(\cdot,z)|_{B_{\delta}(y)}$ and the  uniqueness of martingale problem $(\cL, \delta_y)$. Suppose that  $(\mP_{\mu}, X_t)$ is the unique martingale solution to $(\cL, \mu)$, then for any non-negative function $f\in C_c^\infty$, 
	\begin{equation*}
		\begin{aligned}
			&\mE_{\mu} \int_{0}^{\tau_1} \e^{-\lambda t} f(X_t)\,\d t =\int_{\mR^d} \mu(\d y)~ \mE_y\int_{0}^{\tau_1}   \e^{-\lambda t} f(X_t)\,\d t \\
			=& \int_{\mR^d} \mu(\d y)~\bE_{\mQ_y} \int_0^{\tau_1}  \e^{-\lambda t} f(X_t)\,\d t
			\leq \sup_y \bE_{\mQ_y} \int_0^{\infty}  \e^{-\lambda t} f(X_t)\,\d t \lesssim \|f\|_A.
		\end{aligned}
	\end{equation*}
	By strong Markov property, 
	\begin{equation*}
		\begin{aligned}
			\mE_{\mu} \int_{\tau_k}^{\tau_{k+1}} \e^{-\lambda t} f(X_t)\,\d t=& \mE_{\mu} \int_{\tau_k}^{\tau_{k+1}} \e^{-\lambda t} f(X_t)\,\d t = \mE_{\mu}  \int_{0}^{\tau_1\circ\theta_{\tau_k}} \e^{-\lambda (\tau_k+t)} f(X_t\circ\theta_{\tau_k})\,\d t \\
			=&\mE_{\mu} \e^{-\lambda \tau_k} \l( \int_{0}^{\tau_1} \e^{-\lambda t} f(X_t)\,\d t \r)\circ\theta_{\tau_k}\\
			=&\mE_{\mu} \l( \e^{-\lambda \tau_k} \mE_{X_{\tau_k}}   \int_{0}^{\tau_1} \e^{-\lambda t} f(X_t)\,\d t \r)\\
			\lesssim& \|f\|_A \mE_{\mu} \e^{-\lambda \tau_{k}} \leq \|f\|_A~\sup_{y} \mE_y \e^{-\lambda \tau_{k}}. 
		\end{aligned}
	\end{equation*}
	To estimate $\sup_{y} \mE_y \e^{-\lambda \tau_{k}}$, using strong Markov property again, we have 
	\begin{equation*}
		\begin{aligned}
			\mE_x \e^{-\lambda \tau_{k+1}} =& \mE_x \e^{-\lambda (\tau_{k}+\tau_1\circ\theta_{\tau_{k}})} = \mE_x \l[\e^{-\lambda \tau_k} \mE_{x}\l( \e^{-\lambda \tau_1}\circ\theta_{\tau_k}|\sF_{\tau_k}\r) \r]\\
			=&\mE_x \l[\e^{-\lambda \tau_k} \mE_{X_{\tau_{k}}} \e^{-\lambda \tau_1} \r] \leq \l(\sup_y\mE_y \e^{-\lambda \tau_1}\r) \mE_x \e^{-\lambda \tau_k}\\
			\leq & \cdots \leq \l(\sup_y\mE_y \e^{-\lambda \tau_1}\r)^{k+1}. 
		\end{aligned}
	\end{equation*}
	Thus, 
	\[
	\begin{aligned}
		\mE_{\mu} \int_{0}^{\infty} \e^{-\lambda t} f(X_t)\,\d t\lesssim& \sum_{k=0}^\infty  \mE_{\mu} \int_{\tau_k}^{\tau_{k+1}} \e^{-\lambda t} f(X_t)\,\d t\\
		\lesssim&\|f\|_A \sum_{k=1}^\infty \l( \sup_{y} \mE_y \e^{-\lambda \tau_k} \r)\leq  \|f\|_A \sum_{k=1}^\infty \l(\sup_y\mE_y \e^{-\lambda \tau_1}\r)^{k}. 
	\end{aligned}
	\]
	So the desired conclusion follows if $\mE_y \e^{-\lambda \tau_1}\leq 1/2$ for all $y$. To achieve this, we choose $g\in C_b^\infty$, which  satisfies $g(x)=0$ if $x\in B_{\delta/2}$, $g(x)=1$ if $x\in B_{\delta}$ and $\|\nabla g\|_\infty\leq C \delta^{-1}$. Let $g_y(x)=g(x-y)$. It is easy to verify that there exists a constant $K_\delta<\infty$ such that 
	\[
	\cL g_y(x) -\lambda g_y(x) \leq \| \cL g_y \|_\infty \leq K_\delta.   
	\]
    Note that
	\[
	\mE_y \l[ \e^{-\lambda \tau_1} g_y(X_{\tau_1}) \r] -g_y(y) = \mE_y \int_0^\infty \e^{-\lambda s}(\cL g_y-\lambda g_y)(X_s)\,\d s. 
	\]
	
	Since $g_y(y)=g(0)=0$ and $g_y(X_{\tau_1})=1$, $\mP_y$-a.s., we have  
	\[
	\mE_y \e^{-\lambda \tau_1} \leq K_\delta /\lambda. 
	\]
	Choosing $\lambda\geq2K\vee \lambda_0$, we have $\sup_y\mE_y \e^{-\lambda \tau_1}\leq 1/2$. Therefore, 
	\[
	\mE_{\mu} \int_{0}^{\infty} \e^{-\lambda t} f(X_t)\,\d t\lesssim \|f\|_A \sum_{k=1}^\infty \frac{1}{2^k}\lesssim \|f\|_A, \quad \mbox{ for all 
	} \lambda\geq 2K\vee \lambda_0.  
	\]
	For any $\lambda>0$ and $\mu\in \cP(\mR^d)$, 
	\[
	\begin{aligned}
		\mE_{\mu} \int_{0}^{\infty} \e^{-\lambda t} f(X_t)\,\d t \leq& \sup_{y} \mE_y \int_0^1 f(X_t)\,\d t \cdot \sum_{k=0}^\infty \e^{-\lambda k} \\
		\lesssim& \lambda^{-1} \sup_{y} \mE_y \int_0^\infty \e^{-(2K\vee \lambda_0) t} f(X_t)\,\d t \lesssim \|f\|_A/\lambda. 
	\end{aligned}
	\]
	So we complete our proof. 
\end{proof}

\begin{remark}\label{rmk:SIO}
	For diffusion operators, Stroock and Varadhan showed that if the diffusion coefficients are uniformly elliptic and continuous, then the corresponding martingale problem is well-posed (see for instance, \cite{stroock2007multidimensional}). Two key ingredients proving the uniqueness are: 
	\begin{enumerate}
		\item $(I-\Delta)^{-1}L^p \hookrightarrow L^\infty$,  $p>d/2$; 
		\item the $L^p$ boundedness of the Riesz transforms (singular integrals). 
	\end{enumerate}
	Following the approach of Stroock and Varadhan to obtain a similar result for $\cL$ when the coefficient $a$ is merely uniformly continuous in $x$, one would need to establish the boundedness of singular integrals in the Orlicz space $L_A$, which satisfies $(I+\phi(-\Delta))^{-1} L_A \hookrightarrow L^\infty$. Unfortunately, the following example demonstrates that if $A$  grows rapidly, even the Hilbert transform (the archetypal singular integral operator) may not be bounded in $L_A$: let $d=1$ and $H$ be the Hilbert transform. Set $I=(0,1)$ and $A(t)=\e^{t^2}-1$. Then  
	$$
	H\chi_{I}(x)=\log \l|\frac{x}{(x-1)}\r|. 
	$$
	For each $\lambda>0$,  we have 
	\[
	\int A(|H\chi_I|/\lambda)=\int_0^\infty A'(t) \l|\l\{x: H\chi_{I}(x)>\lambda t\r\}\r| \d t \geq \int_1^\infty \frac{\e^{t^2}}{(\e^{\lambda t}-1)} \d t =\infty. 
	\]
	Thus, it appears that we may need to explore new frameworks to address this problem when $a$ does not satisfies  \ref{Aspt:a-2}. 
\end{remark}

We conclude this section by providing the proof of Corollary \ref{Cor:SDE}.
\begin{proof}[Proof of Corollary \ref{Cor:SDE}]
    The existence of a weak solution to \eqref{Eq:SDE} follows from standard arguments under the assumption that the coefficient is continuous and bounded.
    
    For the uniqueness, let $S$ be the gamma subordinator, whose Laplace exponent is given by  $\phi(s)=\log(1+s)$. Under the assumption that \(\sigma\) is non-degenerate, as presented in Example \ref{ex:sde}, the generator operator \(\cL_\sigma\) associated with \eqref{Eq:SDE} is given by \eqref{Eq:L-sde}. In Section \ref{subsec:subordinator}, we have verified that the jump kernel of the variance gamma process and the function \(a\) satisfy conditions \ref{Aspt:J-1} and \ref{Aspt:a-1}-\ref{Aspt:a-2}, respectively. Since any weak solution to \eqref{Eq:SDE} is also a solution to martingale problem $(\cL_\sigma, \delta_x)$, in the light of Theorem \ref{Thm:Main2}, \eqref{Eq:SDE} admits a unique solution. 
    
    For \eqref{Eq:Krylov2}, set $A(t)=\e^{t^\beta}-1$ with $\beta>1$. Then $[\psi^{-1}(t^{1+\eps})]^d= (\e^{t^{1+\eps}}-1)^{d/2}\lesssim A(t)$, for all $t\gg 1$ and some $\eps\in (0, \beta-1)$. Thanks to Theorem \ref{Thm:Main2} (inequality \eqref{Eq:Krylov1}), we obtain that 
    \[
	\begin{aligned}
            \mE_x \int_0^\infty \e^{-\lambda t} f(X_t)\,\d t\leq& \frac{C}{\lambda} \inf\l\{ \lambda>0: \int_{\mR^d} A\l({f(x)}/{\lambda}\r) \d x\leq 1\r\}\\
            =& \frac{C}{\lambda} \inf\l\{ \lambda>0: \int_{\mR^d} \l( \exp\l[\l(f(x)/\lambda\r)^\beta\r]-1 \r) \d x\leq 1\r\}. 
	\end{aligned}  
    \]
\end{proof}

\appendix
\section{Remarks on a theorem by R. Bass}\label{Appendix-stable}

\setcounter{equation}{0}
\renewcommand\theequation{A.\arabic{equation}}

This section extends the main results of \cite{bass2009regularity} on stable-like operators using Littlewood-Paley theory and scaling techniques, and emphasizes that it is more natural to study a priori estimates in the $\sC^s$ space. We also observe that the function space $X$ mentioned in Section \ref{sec-approach} (aimed to replace $\sC^s$) contains unbounded and discontinuous functions, so it is not suitable for the problems we care about. We will also show that techniques that have proven to be successful for stable-like operators can not yield satisfactory results for our main problem of concern. 

To begin, we review some fundamental concepts from classical Littlewood-Paley theory. Let $\chi$ be the same smooth function defined in Section \ref{Sec-OB}. 
$$
\varphi(\xi):=\chi(\xi)-\chi(2\xi).
$$
It is easy to see that $\varphi\geq 0$,  
\begin{equation}\label{eq:phi=1}
	\mathrm{supp}\, \varphi\subset \cC:=B_{1}\setminus B_{\frac{3}{8}} ~\mbox{ and }~ \varphi(x)=1 \mbox{ if } x\in B_{\frac{3}{4}}\backslash B_{\frac{1}{2}}. 
\end{equation}
Operators $\Delta_j$ is defined by
\begin{align*}
	\Delta_j f:=
	\left\{
	\begin{aligned}
		&\cF ^{-1}(\chi(2\cdot) \cF  f), & \quad j=-1, \\
		&\cF ^{-1}(\varphi(2^{-j}\cdot) \cF  f),& \quad j\geq 0. 
	\end{aligned}
	\right. 
\end{align*}

\begin{definition}[H\"older-Zygmund space]
	Assume that $s\in \mR$, let $\sC^s$ denote the collection of all distribution $f$ satisfying 
	$$
	\|f\|_{\sC^s}:= \sup_{j\geq -1} 2^{js} \|\Delta_j f\|_\infty <\infty. 
	$$
\end{definition}
The following result is well-known. 
\begin{theorem}[{\cite[Theorem  2.36]{bahouri2011fourier}}]\label{Thm:chart0}
	Assume $s>0$ and $s\notin \mN$, then $\sC^s$ is the usual H\"older space and 
	\[
	\|f\|_{\sC^s}\asymp \|f\|_\infty+ \sup_{x\neq y} \frac{|\nabla^{[s]}f(x)-\nabla^{[s]}f(y)|}{|x-y|^{s-[s]}}.
	\]
\end{theorem}

Let $\nu$ be a non-degenerate $\alpha$-stable measure, i.e., for each $\xi\in \mR^d\backslash \{0\}$,
\[
\int_{\mS^{d-1}} |\sigma\cdot \xi|^2 \Sigma(\d \sigma)>0.
\]
Here $\Sigma$ is the spectral measure of $\nu$, i.e. 
\[
\nu(\d z)= \frac{\d r}{r^{1+\alpha}} \Sigma(\d \sigma), \quad z=r\sigma. 
\]
Let $a$ is a bounded positive  function on $\mR^d\times\mR^d$. The operator $\sL_a$ is defined by 
\[
\sL_a u(x)= \int_{\mR^d} \l(u(x+z)-u(x)-\nabla u(x)\cdot z \1_{\alpha\in (1,2)}\r)~a(x,z)\nu(\d z) 
\]
when $\alpha\in (0,2)$ and $\alpha\neq 1$. If $\alpha=1$,  we always assume $\int_{\mS^{d-1}} \sigma ~ a(x,r\sigma)\Sigma(\d \sigma)=0$ and define 
\[
\sL_a u(x)= \int_{\mR^d} \l(u(x+z)-u(x)-\nabla u(x)\cdot z \1_{B_1}(z)\r)~a(x,z)\nu(\d z). 
\]

The following result is a generalization of Proposition 4.2 and Proposition 4.3 in \cite{bass2009regularity}. Note that here we do not need to assume $\nu$ is absolutely continuous with respect to Lebesgue measure. 
\begin{lemma}\label{Lem-A-schauder}
	Let $c_0\in (0,1)$. Suppose $\nu$ is a non-degenerate $\alpha$-stable measure. Assume that $a$ only depends on $z$ and $c_0\leq a\leq c_0^{-1}$, then for each $\lambda \geq 1, \beta\in \mR$, there is a constant $C=C(d, \nu, \alpha, \beta, c_0)$ such that 
	\begin{equation}\label{eq:schauder-stable}
		\lambda \|u\|_{\sC^\beta} + \|u\|_{\sC^{\alpha+\beta}}\leq C \|\lambda u-\sL_{a} u\|_{\sC^\beta}. 
	\end{equation}
\end{lemma}
\begin{proof}
	We only prove the case $\alpha\neq 1$ here. 
	
	\noindent {\em Step 1}. 
	Let 
	\[
	\cA=\l\{a\in L^\infty(\mR^d): c_0\leq a(z)\leq c_0^{-1}, z\in \mR^d\r\}. 
	\]
	We first prove that for any $v\in \sS'(\mR^d)$ satisfying $\mathrm{supp}~\hat{v}\subseteq \cC=B_{1}\setminus B_{\frac{3}{8}}$, it holds that 
	\begin{equation}\label{eq:sLv}
		\|v\|_{\infty} \leq C\inf_{a\in \cA}\|\sL_a v\|_{\infty}. 
	\end{equation}
	Here $C$ is a constant independent of on $v$. 
	
	Assume \eqref{eq:sLv} does not hold. Then, there is a sequence $v_n$ such that 
	$\widehat{v_n}$ is supported on $\cC$, and a sequence $a_n\in \cA$ such that 
	\begin{equation}\label{eq:1-Lvn}
		1=\|v_n\|_\infty\geq n \|\sL_{a_n} v_n\|_\infty. 
	\end{equation}
	Let $h=\sF^{-1}(\chi(\cdot/2)-\chi(4\cdot))$, where $\chi$ is the same function in Section \ref{Sec-OB}. Noting that $(\chi(\cdot/2)-\chi(4\cdot)) \widehat{v_n}=\widehat{v_n}$, we have 
	$$
	v_n(x)=\int_{\mR^d} h(x-y) v_n(y)\d y. 
	$$
	So for any $k\in \mN$,
	\begin{equation}\label{eq:vn-bdd}
		\|\nabla^k v_n\|_\infty=\|\nabla^k h* v_n\|_\infty\leq \|\nabla^k h\|_{1}\|v_n\|_\infty\leq C_k.
	\end{equation}
	By Ascoli-Azela's lemma and  diagonal argument, there is a subsequence of $\{v_n\}$ (still denoted by $v_n$ for simplicity) and $v\in C_b^\infty$ such that $\nabla^k v_n$ converges to $\nabla^k v$ uniformly on any compact set. Let $\chi_R(\cdot)=\chi(\cdot/R)$. For any $\phi\in \sS(\mR^d)$,
	\begin{align*}
		\left|\int \phi(v_n-v)\right|\leq& \int |\phi\chi_R\cdot (v_n-v) | + \int |\phi(1-\chi_R) (v_n-v) |\\
		\leq& \|\phi\|_{L^1} \|v_n-v\|_{L^\infty(B_{3R/2})} + 2 \sup_{|x|>R} |\phi(x)|.
	\end{align*}
	Letting $n\to \infty$ and then $R\to\infty$, we get
	$$
	\<\phi, v_n\>\to \<\phi, u\>, \quad \forall \phi\in \sS(\mR^d).
	$$
	Here $\<,\>$ denotes the dual pair of $\sS(\mR^d)$ and $\sS'(\mR^d)$. That is to say $v_n\to v$ in $\sS'(\mR^d)$ and consequently, $\hat{v_n}\to \widehat{v}$ in $\sS'(\mR^d)$. For any $\phi\in \sS(\mR^d)$ supported on $\mR^d\backslash\cC$, we have
	$$
	\<\phi, \hat{v}\>=\lim_{n\to\infty} \<\phi, \widehat{v_n}\>=0,
	$$
	which means $\hat{v}$ is also supported on $\cC$. 
	
	On the other hand, note that $a_n \nu$ is a sequence of Radon measures on $\mR^d\backslash \{0\}$, and $\sup_{n} \int_{K} a_n(z)\nu(\d z)\leq c_0^{-1} \nu(K)<\infty$ for each compact subset $K$ in $\mR^d\backslash \{0\}$. Thus, there is a subsequence of $a_n$ (still denoted
	by $a_n$ for simplicity) such that
	\begin{equation}\label{eq:an-a0}
		a_n\nu\overset{v}{\longrightarrow}\nu_0 \ \mbox{ and } \ a_0 := \d \nu_0/\d \nu \in [c_0, c_0^{-1}]. 
	\end{equation} 
	Using \eqref{eq:1-Lvn}-\eqref{eq:an-a0} and the fact that $\nabla^k v_n\to \nabla^k v$ uniformly on each compact subset, we have 
	\begin{equation*}
		\begin{aligned}
			\sL_{a_0} v(x)=\ &\int_{\mR^d} \l(v(x+z)-v(x)-\nabla v(x)\cdot z \1_{\alpha\in (1,2)}\r)~a(z)\nu(\d z)\\
			=\ &\lim_{n\to\infty}\l|\int_{\mR^d} \l(v_n(x+z)-v_n(x)-\nabla v_n(x)\cdot z \1_{\alpha\in (1,2)}\r)~a_n(z) \nu(\d z)\r|\\
			\overset{\eqref{eq:1-Lvn}}{\leq}& \lim_{n\to\infty}\frac{\|v_n\|_\infty}{n} 
			=0. 
		\end{aligned}
	\end{equation*}
	This implies  
	$0=\widehat{\sL_{a_0} v}=\psi_{a_0} \hat{v}$. Since for $\xi\neq 0$, 
	\begin{equation*}
		\begin{aligned}
			\Re (\psi_{a_0}(\xi)) =&\int_{\mR^d} (1-\cos(z\cdot \xi))~a_0(z)\nu(\d z)\gtrsim  \int_{|z|\leq \frac{\pi}{4|\xi|}} |z\cdot\xi|^2 \nu(\d z)\\
			=&\int_0^{\frac{\pi}{4|\xi|}}r^{1-\alpha} \d r\int_{\mR^{d-1}} |\sigma\cdot \xi|^2 \Sigma(\d \sigma)>0, 
		\end{aligned}
	\end{equation*} 
	we get that the support of $\hat{v}$ is the origin. This contradicts the previous conclusion that $\mathrm{supp}\  \hat{v}\subseteq \cC$. So we complete the proof for \eqref{eq:sLv}. 
	
	\noindent {\em Step 2}.  
	Suppose $u\in \sS'(\mR^d)$ and $\mathrm{supp}\ \hat{u}\subseteq \lambda \cC$. Let $v(x):= u(\lambda^{-1} x)$ and $a^\lambda(z)=a(\lambda^{-1} z)$. By the scaling property of $\nu$, we get 
	\begin{equation*}
		\begin{aligned}
			\sL_{a^{\lambda}} v(x)=&
			\int_{\mR^d}\l(u(\lambda^{-1} x+\lambda^{-1} z)- u(\lambda^{-1} x)-\nabla u(\lambda^{-1}(x)) \cdot (\lambda^{-1} z) \1_{\alpha\in(1,2)} \r) ~ a^{\lambda}(z)\nu(\d z)\\
			=& 
			\int_{\mR^d}\l(u(\lambda^{-1} x+z')- u(\lambda^{-1} x)-\nabla u(\lambda^{-1}(x)) \cdot z' \1_{\alpha\in(1,2)} \r)~a(z')\lambda^{-\alpha}\nu(\d z')\\
			=& \lambda^{-\alpha}(\sL_a u)(\lambda^{-1} x). 
		\end{aligned}
	\end{equation*}
	Thus, by \eqref{eq:sLv}, we get 
	\begin{equation}\label{eq:u-Lu}
		\|u\|_\infty= \|v\|_\infty\leq C  \|\sL_{a^{\lambda}} v\|_\infty =C\lambda^{-\alpha} \|\sL_a u\|_\infty, 
	\end{equation}
	provided that $\mathrm{supp}\ \hat{u}\in \lambda \cC$. 
	
	\noindent {\em Step 3}. Like the proof for Theorem \ref{Thm:Main0}, it is easy to see that for each $j\geq -1$,
	\[ 
	\lambda \|\Delta_j u\|_\infty \leq \|\lambda \Delta_j u-\sL_a \Delta_j u\|_\infty. 
	\]
	By \eqref{eq:u-Lu}, for any $j\geq 0$, 
	\begin{align*}
		2^{j\alpha}\|\Delta_j u\|_\infty\leq & C \|\sL_a \Delta_j u\|_\infty \leq C \lambda \|\Delta_j u\|_\infty +  C \|\lambda \Delta_j u-\sL_a \Delta_j u\|_\infty\\
		\leq &  \|\lambda \Delta_j u-\sL_a \Delta_j u\|_\infty. 
	\end{align*}
	Thus, for all $\lambda \geq 1$ and $\beta\in \mR$, we get \eqref{eq:schauder-stable}. 
\end{proof}

\begin{theorem}
    Let $\nu$ be a non-degenerate $\alpha$-stable measure with \(\alpha\in (0,2)\). Suppose that there exist constants \(c_0\) and \(\gamma\in (0,1)\), such that $c_0\leq a\leq c_0^{-1}$ and \(\sup_{z\in \mR^d}\|a(\cdot,z)\|_{\sC^\gamma}\leq c_0^{-1}\), then for any $\lambda >0$ and $\beta\in [0,\gamma]$, there is a constant $C=C(d, \nu, \alpha, \beta, \gamma, c_0)$ such that 
    \begin{equation*}
        \lambda \|u\|_{\sC^\beta} + \|u\|_{\sC^{\alpha+\beta}}\leq C \|\lambda u-\sL_{a} u\|_{\sC^\beta}. 
    \end{equation*}
\end{theorem}
With Lemma \ref{Lem-A-schauder} at hand, utilizing the method of frozen coefficients as before, we can prove the case that $a$ is a positive function on $\mR^d\times\mR^d$.  Since the proof for this result is just a repetition of the proof of Theorem \ref{Thm:Main1}, we omit it here.

\smallskip

We point out that the approach used in the proof of Lemma \ref{Lem-A-schauder} cannot be directly extended to the case of $\nu(\mathrm{d}z)=\frac{\1_{B_1}(z)}{|z|^d}\mathrm{d}z$, and it may not be possible. One significant hurdle arises from the need to obtain the desired regularity estimate, which requires the use of the $\psi$-decomposition with $\psi(R)\asymp \log R\ (R\to\infty)$. However, when attempting to repeat the same procedure, it becomes apparent that {\em Step 2} in the proof for Lemma \ref{Lem-A-schauder} is invalid in the current situation.

\section{A remark on space \texorpdfstring{$\sX^s$}{}}\label{Appendix-exmple}

\setcounter{equation}{0}
\renewcommand\theequation{B.\arabic{equation}}

In this section, we show that $\sX^s$ given by \eqref{Eq:X} may contains unbounded discontinuous functions, even if $s>0$. 

\begin{proposition}
	Let $\sX^1$ be the function space given by \eqref{Eq:X}. Then $\sX^1\not\subseteq L^\infty$. 
\end{proposition}
\begin{proof}
	We provide an example to demonstrate that  $\sX^1$ encompasses unbounded,  discontinuous functions. 
	
	Let $\phi\in \sS(\mR^d)$ such that $\cF(\phi)\in C_c^\infty(B_{\frac{3}{4}}\backslash B_{\frac{1}{2}})$ , $\cF(\phi)\in [0,1]$ and $\cF(\phi)=1$ on $B_{\frac{11}{16}}\backslash B_{\frac{9}{16}}$. Define 
	\begin{align*}
		f(x)=\left\{
		\begin{aligned}
			&\sum_{j\geq 1} j^{-1} \phi(2^jx),& \quad & x\neq 0\\
			&0,& \quad &x=0. 
		\end{aligned} 
		\right.
	\end{align*}
	Due to \eqref{eq:phi=1}, $\cF(\phi) \varphi = \cF(\phi)$.  Thus,  
	\begin{align*}
		\cF(f)(\xi)\varphi(2^{-j}\xi)=&j^{-1}2^{-j}(\cF(\phi)\varphi)(2^{-j}\xi)\\
		=&j^{-1}2^{-j}\cF(\phi)(2^{-j}\xi).  
	\end{align*}
	This yields    
	\[
	\sup_{j\geq -1} \max\{1, j\}\|\Delta_j f\|_\infty=  \sup_{j\geq 1} \|\phi (2^j \cdot)\| = \|\phi\|_\infty<\infty.
	\]
	Therefore, $f\in \sX^1$. 
	
	On the other hand, since $\phi\in\sS(\mR^d)$ and 
	\[
	|\phi(x)|=\l| \int_{\mR^d} \cF(\phi)(\xi) \e^{i 2\pi x\cdot \xi} \d \xi \r| \leq \int_{\mR^d} \cF(\phi)(\xi) \d \xi=\phi(0)>0, 
	\]
	there are constants $\eps_0, \delta_0\in (0,2^{-10})$ such that $\phi(x)\geq 2\eps_0, \mbox{ when } |x|\leq \delta_0$. Now put $\delta:=\delta_0^{10}$. For any $x\in B_\delta\backslash \{0\}$, set $j_x:= [-\log_2 |x|]$. It holds that 
	\begin{equation}\label{eq:f1}
		C\log j_x \geq \phi(0) \sum_{1\leq j\leq j_x} j^{-1}\geq  \sum_{1\leq j\leq j_x+\log_2 \delta_0} j^{-1} \phi(2^jx) \geq \eps_0 \log j_x. 
	\end{equation}
	Noting that $|\phi(2^j x)|\leq C (2^j |x|)^{-N}$ if $j\geq j_x$, we get 
	\begin{equation}\label{eq:f2}
		\sum_{j> j_x} |j^{-1} \phi(2^jx)| \leq C \sum_{j> j_x} j^{-1} 2^{-jN} |x|^{-N} \leq  C j_x^{-1}\ll 1.
	\end{equation}
	Moreover, 
	\begin{equation}\label{eq:f3}
		\sum_{j_x+\log_2\delta_0<j\leq j_x} |j^{-1} \phi(2^jx)| \leq -C\log_2 \delta_0 . 
	\end{equation}
	Combining \eqref{eq:f1}-\eqref{eq:f3}, one sees that for any $x\in B_\delta\backslash\{0\}$, 
	\[
	f(x)\geq \eps_0 \log j_x -C \geq \frac{\eps_0}{2} \log\log_2 |x|^{-1}-C 
	\]
	and 
	\[
	|f(x)|\leq C \log \log_2 |x|^{-1}+C. 
	\]
	This implies $f\notin L^\infty$ and $f\in L^1(B_\delta)$. 
	
	To sum up, we see that $f\in (\sX^1\cap L^1_{\mathrm{loc}})\backslash L^\infty$. 
\end{proof}

\medskip

\paragraph{\bf Acknowledgements }The authors would like to thank Huyuan Chen and Jaehoon Kang for many useful conversations. 

\paragraph{\bf Foundings }Research Eryan is supported by National Key R\&D Program of China (No. 2022YFA1006000) and by the National Natural Science Foundation of China (No. 12171354); Research Guohuan is supported by the National Natural Science Foundation of China (Nos. 12288201, 12271352, 12201611)

\paragraph{\bf Data availability }The manuscript has no associated data.

\bigskip

\paragraph{\Large{\bf{Declarations}}}

\paragraph{\bf Conflict of interest} On behalf of all authors, the corresponding author states that there is no conflict of
interest.

\end{document}